\documentclass{amsart}
\usepackage{lahn}
\usepackage[ style = alphabetic ]{biblatex}

\title{Which Reducible Representations are Anosov?}
\author{Max Lahn}
\address{Department of Mathematics \\ University of Michigan \\ Ann Arbor, MI 48109}
\email{\href{mailto:maxlahn@umich.edu}{maxlahn@umich.edu}}

\DeclareDocumentCommand{\K}{}{ \mathbb{ K } }

\DeclareDocumentCommand{\Char}{}{ \mathcal{ X } }

\begin{document}

\begin{abstract}

We give a characterization of the Anosov condition for reducible representations in terms of the eigenvalue magnitudes of the irreducible block factors of its block diagonalization. As in previous work, these Anosov representations comprise a collection of bounded convex domains in a finite-dimensional vector space, and this perspective allows us to conclude for many non-elementary hyperbolic groups that connected components of the character variety which consist entirely of Anosov representations do not contain reducible representations.

\end{abstract}

\maketitle

\section{Introduction} \label{sec:introduction}

Discrete subgroups of the projective special linear group $ \PSL \of{ 2 , \K } $ are called \emph{Fuchsian} when $ \K = \R $ and \emph{Kleinian} when $ \K = \C $. The rich theory of such groups is central in the geometry, topology, and dynamics of Riemann surfaces and hyperbolic $ 3 $-manifolds, respectively, and has connections to complex dynamics, number theory, and mathematical physics. Work of Labourie \cite{Lab06}, Guichard--Wienhard \cite{GW12}, Bridgeman--Canary--Labourie--Sambarino \cite{BCLS15}, Gu\'{e}ritaud--Guichard--Kassel--Wienhard \cite{GGKW17}, and Kapovich--Leeb--Porti \cite{KLP17,KLP18} has established the theory of Anosov representations as a suitable analogue of convex cocompact Fuchsian and Kleinian groups in the setting of representations of $ \Gamma $ in higher rank Lie groups.

Further parallels to the convex cocompact Fuchsian and Kleinian cases are possible when the images of the Anosov representations under consideration are dense in the Zariski topology. However, Anosov representations in higher rank Lie groups need not be Zariski-dense, and the deformation theory of Anosov representions often must contend with the existence of such representations. For example, when $ \dim_{ \K } \of{ V } \geq 3 $, one may exploit an irreducible representation $ \SL_{ 2 } \of{ \K } \into \SL \of{ V } $ to obtain representations in $ \SL \of{ V } $ which are not Zariski dense from convex cocompact Fuchsian and Kleinian groups, and the deformations of such representations are of great interest. In the case $ \K = \R $, these comprise the Hitchin component of the character variety, and in the case $ \K = \C $, they contain the quasi-Fuchsian space.

In previous work \cite{Lah24}, we explored variations on this phenomenon by considering certain deformations of \emph{reducibly} embedded Fuchsian/Kleinian subgroups. Such deformations were parameterized by a finite-dimensional real vector space, and those corresponding to Anosov representations comprised a convex and bounded open subset. This article serves to generalize those results from reducible suspensions of $ q $-Fuchsian representations to all reducible representations and to make initial conclusions about the character variety.

Concretely, fix a non-elementary word hyperbolic group $ \Gamma $ and a finite-dimensional vector space $ V $ over a field $ \K $, either $ \R $ or $ \C $. Let $ \mathscr{ U } $ be a direct sum decomposition of $ V $; that is, $ \mathscr{ U } = \of{ U_{ j } }_{ j = 1 }^{ m } $ is a finite sequence of pairwise trivially intersecting subspaces $ U_{ j } \subseteq V $ which jointly span $ V $. As described in \autoref{sec:background}, the partial sums of the factors of this decomposition comprise a partial flag in $ V $. Representations of $ \Gamma $ in $ \GL \of{ V } $ which preserve each factor subspace $ U_{ j } $ are called \emph{block diagonal} relative to $ \mathscr{ U } $, and those which preserve the associated partial flag are called \emph{block upper triangular} relative to $ \mathscr{ U } $.

We first adapt Kassel--Potrie's \cite{KP22} characterization of the Anosov condition in terms of the linear growth of certain functions of the Jordan projection to the setting of block upper triangular representations. Anosov representations in $ \GL \of{ V } $ come in several flavors $ P_{ \theta } $, one for each non-empty subset $ \theta $ of $ \Delta_{ V } = \set{ 1 , \dotsc , \dim_{ \K } \of{ V } - 1 } $. We show that to each block upper triangular $ P_{ \theta } $-Anosov representation is associated a family of integers, called the \emph{large eigenvalue $ \theta $-configuration} relative to $ \mathscr{ U } $, which catalogues the distribution of generalized eigenvectors of large eigenvalue among the preserved partial flag's subspaces. For ease of exposition, we emphasize here the simpler result in the block diagonal setting; the general result for block upper triangular representations is \autoref{thm:eigenvalueConfiguration}.

\begin{theorem} \label{thm:newEigenvalueConfiguration}

Fix a representation $ \eta \colon \Gamma \to \GL \of{ V } $ and a subset $ \theta \subseteq \Delta_{ V } $. If $ \eta $ is block diagonal relative to $ \mathscr{ U } $, with block decomposition
\[
\eta = \bigoplus_{ j = 1 }^{ m }{ \eta_{ j } } ,
\]
then $ \eta $ is $ P_{ \theta } $-Anosov if and only if there is a family $ Q = \of{ q_{ j , k } }_{ j , k } $ of integers $ 0 \leq q_{ j , k } \leq \dim_{ \K } \of{ U_{ j } } $, indexed by $ 1 \leq j \leq m $ and $ k \in \theta $, with the following properties:
\begin{enumerate}

\item[(i)]

$ \sum_{ j = 1 }^{ m }{ q_{ j , k } } = k $ for all $ k \in \theta $.

\item[(ii)]

$ \log \of{ \frac{ \lambda_{ q_{ i , k } } \of{ \eta_{ i } \of{ \gamma } } }{ \lambda_{ q_{ j , k } + 1 } \of{ \eta_{ j } \of{ \gamma } } } } $ grows at least linearly in translation length $ \norm{ \gamma }_{ \Sigma } $ for all $ k \in \theta $ and $ 1 \leq i , j \leq m $ so that $ q_{ i , k } > 0 $ and $ q_{ j , k } < \dim_{ \K } \of{ U_{ j } } $.

\end{enumerate}
In this case, such a family $ Q $ is uniquely defined, and the following also hold:
\begin{enumerate}

\item[(1)]

The $ P_{ \theta } $-Anosov limit map $ \xi_{ \eta }^{ \theta } \colon \partial \Gamma \to \Flag_{ \theta } \of{ V } $ has the property that
\[
q_{ j , k } = \dim_{ \K } \of{ U_{ j } \cap \xi_{ \eta }^{ \theta } \of{ z } }
\]
for all $ 1 \leq j \leq m $, $ k \in \theta $, and $ z \in \partial \Gamma $.

\item[(2)]

For each $ 1 \leq j \leq m $, let $ \theta_{ j } = \set{ q_{ j , k } }_{ k \in \theta } \cap \Delta_{ U_{ j } } $. If $ \theta_{ j } \neq \emptyset $, then $ \eta_{ j } $ is $ P_{ \theta_{ j } } $-Anosov, with $ P_{ \theta } $-Anosov limit map $ \xi_{ \eta_{ j } }^{ \theta_{ j } } \colon \partial \Gamma \to \Flag_{ \theta_{ j } } \of{ U_{ j } } $ given by the formula
\[
\xi_{ \eta_{ j } }^{ \theta_{ j } } \of{ z } = \of{ U_{ j } \cap \xi_{ \eta }^{ k_{ q } } \of{ z } }_{ q \in \theta_{ j } } ,
\]
where for each $ q \in \theta_{ j } $, we choose any $ k_{ q } \in \Delta_{ V } $ so that $ q_{ j , k_{ q } } = q $.

\item[(3)]

The block diagonal representation $ \eta $ has $ P_{ \theta } $-Anosov limit map $ \xi_{ \eta }^{ \theta } \colon \partial \Gamma \to \Flag_{ \theta } \of{ V } $ given by the formula
\[
\xi_{ \eta }^{ \theta } \of{ z } = \of{ \bigoplus_{ q_{ i , k } > 0 }{ \xi_{ \eta_{ i } }^{ q_{ i , k } } \of{ z } } }_{ k \in \theta } .
\]

\end{enumerate}

\end{theorem}

Block upper triangular representations are equivalent in the character variety to block diagonal representations. The block diagonal representations relative to $ \mathscr{ U } $ may be parameterized by homomorphisms $ \delta \colon \Gamma \to \R $ and $ \varphi \colon \Gamma \to D_{ \mathscr{ U } } $ and representations $ \zeta \colon \Gamma \to N_{ \mathscr{ U } } $, where $ D_{ \mathscr{ U } } $ is the finite-dimensional real vector space
\[
D_{ \mathscr{ U } } = \set{ \vect{ x } \in \R^{ m } : \sum_{ j = 1 }^{ m }{ \dim_{ \K } \of{ U_{ j } } x_{ j } } = 0 , \textrm{ and } x_{ j } = 0 \textrm{ for all } 1 \leq j \leq m \textrm{ so that } U_{ j } = \set{ \vect{ 0 } } } ,
\]
and $ N_{ \mathscr{ U } } $ is the matrix group
\[
N_{ \mathscr{ U } } = \set{ \bigoplus_{ j = 1 }^{ m }{ \vect{ B }_{ j } } \in \Aut_{ \K } \of{ \mathscr{ U } } : \abs{ \det \of{ \vect{ B }_{ j } } } = 1 \textrm{ for all } 1 \leq j \leq m } .
\]
Concretely, given such homomorphisms $ \delta \colon \Gamma \to \R $ and $ \varphi \colon \Gamma \to D_{ \mathscr{ U } } $ and a representation $ \zeta \colon \Gamma \to N_{ \mathscr{ U } } $, we obtain a block diagonal representation $ \eta = \beta_{ \mathscr{ U } } \of{ \delta , \varphi , \zeta } \colon \Gamma \to \GL \of{ V } $ defined by the formula
\[
\eta \of{ \gamma } = \bigoplus_{ j = 1 }^{ m }{ e^{ \delta \of{ \gamma } + \varphi_{ j } \of{ \gamma } } \zeta_{ j } \of{ \gamma } } . 
\]
We will call homomorphisms $ \Gamma \to D_{ \mathscr{ U } } $ and representations $ \Gamma \to N_{ \mathscr{ U } } $ \emph{block deformations} and \emph{block normalized representations}, respectively. The Anosov conditions for $ \eta = \beta_{ \mathscr{ U } } \of{ \delta , \varphi , \zeta } $ can be seen to be independent of $ \delta $, so that for a fixed subset $ \theta \subseteq \Delta_{ V } $ and block normalized representation $ \zeta \colon \Gamma \to N_{ \mathscr{ U } } $, we construct
\[
A^{ \mathscr{ U } }_{ \theta } \of{ \zeta } = \set{ \varphi \in \hom \of{ \Gamma , D_{ \mathscr{ U } } } : \beta_{ \mathscr{ U } } \of{ \delta , \varphi , \zeta } \textrm{ is $ P_{ \theta } $-Anosov for all (equivalently, any) } \delta \in \hom \of{ \Gamma , \R } } .
\]
That is, $ A^{ \mathscr{ U } }_{ \theta } \of{ \zeta } $ is the space of block deformations $ \varphi $ of the block normalized representation $ \zeta $ which result in a $ P_{ \theta } $-Anosov block diagonal representation $ \beta_{ \mathscr{ U } } \of{ \delta , \varphi , \zeta } $. If the block normalized representation $ \zeta \colon \Gamma \to N_{ \mathscr{ U } } $ is itself $ P_{ \theta } $-Anosov, then we are able to concretely characterize $ A^{ \mathscr{ U } }_{ \theta } \of{ \zeta } $ in terms of $ \varphi $ and $ \zeta $, in analogy to \cite[Theorem~2]{Lah24}.

\begin{restatable}{theorem}{AnosovCharacterization} \label{thm:AnosovCharacterization}

Fix a block normalized representation $ \zeta \colon \Gamma \to N_{ \mathscr{ U } } $ and a subset $ \theta \subseteq \Delta_{ V } $. If $ \zeta $ is $ P_{ \theta } $-Anosov, with large eigenvalue $ \theta $-configuration $ Q = \of{ q_{ j , k } }_{ j , k } $ relative to $ \mathscr{ U } $, then
\[
A^{ \mathscr{ U } }_{ \theta } \of{ \zeta } = \set{ \varphi \in \hom \of{ \Gamma , D_{ \mathscr{ U } } } : \sup_{ i , j , k , \gamma }{ \frac{ \varphi_{ j } \of{ \gamma } - \varphi_{ i } \of{ \gamma } }{ \log \of{ \frac{ \lambda_{ q_{ i , k } } \of{ \zeta_{ i } \of{ \gamma } } }{ \lambda_{ q_{ j , k } + 1 } \of{ \zeta_{ j } \of{ \gamma } } } } } } < 1 } .
\]
where the above supremum is indexed by all $ 1 \leq i , j \leq m $ and $ k \in \theta $ so that $ q_{ i , k } > 0 $ and $ q_{ j , k } < \dim_{ \K } \of{ U_{ j } } $, and infinite order $ \gamma \in \Gamma $.

\end{restatable}

One can see that if a block normalized representation $ \zeta \colon \Gamma \to N_{ \mathscr{ U } } $ is $ P_{ \theta } $-Anosov, then \autoref{thm:AnosovCharacterization} expresses $ A^{ \mathscr{ U } }_{ \theta } \of{ \zeta } $ as an increasing union of intersections of half-spaces, so that $ A^{ \mathscr{ U } }_{ \theta } \of{ \zeta } $ is convex. Moreover, since the above supremum is a homogeneous function of $ \varphi $, $ A^{ \mathscr{ U } }_{ \theta } \of{ \zeta } $ is also bounded. In fact, these properties hold regardless of whether or not $ \zeta $ is $ P_{ \theta } $-Anosov, in analogy to \cite[Theorem~1]{Lah24}.

\begin{restatable}{theorem}{domain} \label{thm:domain}

$ A^{ \mathscr{ U } }_{ \theta } \of{ \zeta } $ is a convex and bounded open subset $ \hom \of{ \Gamma , D_{ \mathscr{ U } } } $ for any direct sum decomposition $ \mathscr{ U } $ of $ V $, non-empty subset $ \emptyset \subsetneq \theta \subseteq \Delta_{ V }$, and block normalized representation $ \zeta \colon \Gamma \to N_{ \mathscr{ U } } $.

\end{restatable}

One consequence of \autoref{thm:domain} is a group-theoretic condition on the hyperbolic groups whose character varieties have Anosov components containing reducible representations.

\begin{restatable}{corollary}{teichmuller} \label{cor:teichmuller}

Choose $ G $ and $ H $ from among $ \SL \of{ V } $, $ \SL^{ * } \of{ V } $, and $ \GL \of{ V } $. If the commutator subgroup $ \comm{ \Gamma }{ \Gamma } $ has infinite index in $ \Gamma $, then any connected component of the character variety $ \hom \of{ \Gamma , G } \mathbin{ // } H $ consisting entirely of ($ H $-equivalence classes of) Anosov representations contains no ($ H $-equivalence classes of) reducible representations.

\end{restatable}

We note that if a block normalized representation $ \zeta \colon \Gamma \to N_{ \mathscr{ U } } $ is discrete and faithful (resp. has finite kernel), then all deformations of the form $ \beta_{ \mathscr{ U } } \of{ \delta , \varphi , \zeta } $ for some homomorphisms $ \delta \colon \Gamma \to \R $ and $ \varphi \colon \Gamma \to D_{ \mathscr{ U } } $ are also discrete and faithful (resp. have finite kernel). Thus \autoref{thm:domain} and \autoref{cor:teichmuller} do not imply that all connected components of the character variety consisting of equivalence classes of discrete and faithful representations contain no equivalence classes of reducible representations.

\subsection*{Acknowledgments}

The author gratefully acknowledges the support of NSF grant DMS-2304636, and the guidance, feedback, and support of Dick Canary and Konstantinos Tsouvalas.

\section{Background, Notation, and Small Lemmas} \label{sec:background}

\subsection{Representations and Character Varieties}

Fix forever a non-elementary word hyperbolic group $ \Gamma $ and a field $ \K $, either $ \R $ or $ \C $. A \emph{representation} of $ \Gamma $ in a group $ G \leq \GL \of{ V } $ of linear automorphisms of a finite-dimensional vector space $ V $ over $ \K $ is a homomorphism $ \Gamma \to G $. The set $ \hom \of{ \Gamma , G } $ of such representations is topologized by the compact open topology, which is also the topology inherited as a subset of $ G^{ \Sigma } $ for any finite generating set $ \Sigma $ for $ \Gamma $. Endowed with this topology, $ \hom \of{ \Gamma , G } $ is called the \emph{representation variety} of $ \Gamma $ in $ G $. Such a representation variety $ \hom \of{ \Gamma , G } $ comes equipped with a left action of the normalizer
\[
N_{ \GL \of{ V } } \of{ G } = \set{ \vect{ A } \in \GL \of{ V } : \vect{ A } G \vect{ A }^{ -1 } = G } = \set{ \vect{ A } \in \GL \of{ V } : \vect{ A } \vect{ B } \vect{ A }^{ -1 } \in G \textrm{ for all } \vect{ B } \in G }
\]
by post-composition with conjugation; for a linear automorphism $ \vect{ A } \in N_{ \GL \of{ V } } \of{ G } $ and a representation $ \rho \colon \Gamma \to G $, $ \vect{ A } \cdot \rho \colon \Gamma \to G $ is the representation of $ \Gamma $ in $ G $ defined by the formula $ \of{ \vect{ A } \cdot \rho } \of{ \gamma } = \vect{ A } \rho \of{ \gamma } \vect{ A }^{ -1 } $. Although this action is by homeomorphisms, the associated quotient of $ \hom \of{ \Gamma , G } $ typically has some topological pathologies. For example, since orbits are generally not closed, this quotient is generally not Hausdorff. However, a suitably nice quotient may be found by considering the following coarser equivalence relation.

\begin{definition*}[Equivalence of representations; character variety]

Fix a group $ G \leq \GL \of{ V } $ of linear automorphisms of a finite-dimensional vector space $ V $ over $ \K $ and a subgroup $ H \leq N_{ \GL \of{ V } } \of{ G } $. Two representations $ \rho_{ 1 } , \rho_{ 2 } \colon \Gamma \to G $ are called \emph{$ H $-equivalent} if the closures of their orbits under the the restriction of the above action to an action of $ H $ on $ \hom \of{ \Gamma , G } $ intersect, in which case we write $ \rho_{ 1 } \sim_{ H } \rho_{ 2 } $.

$ H $-equivalence $ \sim_{ H } $ is an equivalence relation on the representation variety $ \hom \of{ \Gamma , G } $, whose quotient is called a \emph{character variety} and is denoted $ \Char_{ H } \of{ \Gamma , G } = \hom \of{ \Gamma , G } / { \sim_{ H } } = \hom \of{ \Gamma , G } \mathbin{ // } H $. The associated quotient map is denoted $ \class{ \oparg }_{ H } \colon \hom \of{ \Gamma , G } \onto \Char_{ H } \of{ \Gamma , G } $.

\end{definition*}

We note that for groups $ K \leq G \leq \GL \of{ V } $ of linear automorphisms of a finite-dimensional vector space $ V $ over $ \K $ and a subgroup $ H \leq N_{ \GL \of{ V } } \of{ G } \cap N_{ \GL \of{ V } } \of{ K } $, there is a unique injective and continuous map $ \Char_{ H } \of{ \Gamma , K } \into \Char_{ H } \of{ \Gamma , G } $ so that the following diagram commutes:
\[
\begin{tikzcd}
\hom \of{ \Gamma , K } \ar[ r , hook ] \ar[ d , two heads , " \class{ \oparg }_{ H } " swap ] & \hom \of{ \Gamma , G } \ar[ d , two heads , " \class{ \oparg }_{ H } " ] \\
\Char_{ H } \of{ \Gamma , K } \ar[ r , hook ] & \Char_{ H } \of{ \Gamma , G }
\end{tikzcd}
\]
We will use these maps to pretend that $ \Char_{ H } \of{ \Gamma , K } \subseteq \Char_{ H } \of{ \Gamma , G } $ as subsets of $ \Char_{ H } \of{ \Gamma , \GL \of{ V } } $.

For each finite-dimensional vector space $ V $ over $ \K $, denote by
\[
\SL^{ * } \of{ V } = \set{ \vect{ A } \in \GL \of{ V } : \abs{ \det \of{ \vect{ A } } } = 1 }
\]
the group of linear automorphisms of unit magnitude determinant.

\begin{lemma+}

Fix a normal group $ G \nsleq \GL \of{ V } $ of linear automorphisms of a finite-dimensional vector space $ V $ over $ \K $. Then $ \sim_{ \GL \of{ V } } = \sim_{ \SL^{ * } \of{ V } } $ and $ \sim_{ \GL^{ + } \of{ V } } = \sim_{ \SL \of{ V } } $ on $ \hom \of{ \Gamma , G } $, so that $ \Char_{ \GL \of{ V } } \of{ \Gamma , G } = \Char_{ \SL^{ * } \of{ V } } \of{ \Gamma , G } $ and $ \Char_{ \GL^{ + } \of{ V } } \of{ \Gamma , G } = \Char_{ \SL \of{ V } } \of{ \Gamma , G } $.

\begin{proof}

Note that $ \GL \of{ V } = \GL^{ + } \of{ V } = \SL^{ * } \of{ V } = \SL \of{ V } $ if $ \dim_{ \K } \of{ V } = 0 $, so we may assume that $ \dim_{ \K } \of{ V } > 0 $. Consider the homomorphism $ N_{ V } \colon \GL \of{ V } \onto \SL^{ * } \of{ V } $ defined by the formula
\[
N_{ V } \of{ \vect{ A } } = \frac{ \vect{ A } }{ \abs{ \det \of{ \vect{ A } } }^{ \frac{ 1 }{ \dim_{ \K } \of{ V } } } } .
\]
Note that $ N_{ V } $ is a continuous right inverse to the inclusion $ \SL^{ * } \of{ V } \into \GL \of{ V } $, and that the restriction of $ N_{ V } $ to $ \GL^{ + } \of{ V } $ is a continuous right inverse to the inclusion $ \SL \of{ V } \into \GL^{ + } \of{ V } $. Moreover,
\[
\of{ N_{ V } \of{ \vect{ A } } \cdot \rho } \of{ \gamma } = \of{ \frac{ \vect{ A } }{ \abs{ \det \of{ \vect{ A } }^{ \frac{ 1 }{ \dim_{ \K } \of{ V } } } } } } \rho \of{ \gamma } \of{ \abs{ \det \of{ \vect{ A } } }^{ \frac{ 1 }{ \dim_{ \K } \of{ V } } } \vect{ A }^{ -1 } } = \vect{ A } \rho \of{ \gamma } \vect{ A }^{ -1 } = \of{ \vect{ A } \cdot \rho } \of{ \gamma }
\]
for all linear automorphisms $ \vect{ A } \in \GL \of{ V } $, representations $ \rho \colon \Gamma \to G $, and $ \gamma \in \Gamma $, so that $ N_{ V } \of{ \vect{ A } } \cdot \rho = \vect{ A } \cdot \rho $. We will use this fact in the following argument that two representations $ \rho , \eta \colon \Gamma \to G $ are $ \GL \of{ V } $-equivalent if and only if they are $ \SL^{ * } \of{ V } $-equivalent.

Since $ \SL^{ * } \of{ V } \subseteq \GL \of{ V } $, if $ \rho \sim_{ \SL^{ * } \of{ V } } \eta $, then there are sequences $ \of{ \vect{ A }_{ n } }_{ n = 0 }^{ \infty } $ and $ \of{ \vect{ B }_{ n } }_{ n = 0 }^{ \infty } $ of linear automorphisms $ \vect{ A }_{ n } , \vect{ B }_{ n } \in \SL^{ * } \of{ V } $ so that
\[
\lim_{ n \to \infty }{ \vect{ A }_{ n } \cdot \rho } = \lim_{ n \to \infty }{ \vect{ B }_{ n } \cdot \eta } .
\]
Observing that $ \vect{ A }_{ n } , \vect{ B }_{ n } \in \GL \of{ V } $ for all $ n \in \N $, we conclude that $ \rho \sim_{ \GL \of{ V } } \eta $.

On the other hand, if $ \rho \sim_{ \SL^{ * } \of{ V } } \eta $, then there are sequences $ \of{ \vect{ A }_{ n } }_{ n = 0 }^{ \infty } $ and $ \of{ \vect{ B }_{ n } }_{ n = 0 }^{ \infty } $ of linear automorphisms $ \vect{ A }_{ n } , \vect{ B }_{ n } \in \GL \of{ V } $ so that the above equation of limits holds. But of course
\[
\lim_{ n \to \infty }{ N_{ V } \of{ \vect{ A }_{ n } } \cdot \rho } = \lim_{ n \to \infty }{ \vect{ A }_{ n } \cdot \rho } = \lim_{ n \to \infty }{ \vect{ B }_{ n } \cdot \eta } = \lim_{ n \to \infty }{ N_{ V } \of{ \vect{ B }_{ n } } \cdot \eta } .
\]
Observing that $ N_{ V } \of{ \vect{ A }_{ n } } , N_{ V } \of{ \vect{ B }_{ n } } \in \SL^{ * } \of{ V } $ for all $ n \in \N $, we conclude that $ \rho \sim_{ \SL^{ * } \of{ V } } \eta $.

We have shown that $ \sim_{ \GL \of{ V } } = \sim_{ \SL^{ * } \of{ V } } $. Since we may replace $ \GL \of{ V } $ and $ \SL^{ * } \of{ V } $ in the above argument with $ \GL^{ + } \of{ V } $ and $ \SL \of{ V } $, respectively, we also conclude that $ \sim_{ \GL^{ + } \of{ V } } = \sim_{ \SL \of{ V } } $. \qedhere

\end{proof}

\end{lemma+}

Two equivalent representations share many of the same dynamical properties. For example, we will soon see that one is Anosov if and only if the other is. Towards that goal, to each linear transformation $ \vect{ A } \colon V \to V $ of a finite-dimensional vector space $ V $ over $ \K $ are associated $ \dim_{ \K } \of{ V } $ \emph{eigenvalue magnitudes}
\[
\lambda_{ 1 } \of{ \vect{ A } } \geq \lambda_{ 2 } \of{ \vect{ A } } \geq \dotsb \geq \lambda_{ \dim_{ \K } \of{ V } } \of{ \vect{ A } } ,
\]
which are the magnitudes of the complex roots of the characteristic polynomial of $ \vect{ A } $, counted with algebraic multiplicity and listed in non-decreasing order. If $ \K = \C $, then each of these roots has a generalized eigenvector in $ V $. However, if $ \K = \R $, then any non-real complex eigenvalues will only have generalized eigenvectors in the complexification $ V \otimes_{ \R } \C $ of $ V $.

One can check that the eigenvalue magnitudes define continuous and conjugation-invariant functions $ \lambda_{ 1 } , \dotsc , \lambda_{ \dim_{ \K } \of{ V } } \colon \End_{ \K } \of{ V } \to \clopint{ 0 }{ \infty } $ on the space $ \End_{ \K } \of{ V } $ of linear transformations of $ V $. From this, one can derive the following fact:

\begin{fact+} \label{fact:equivalenceEigenvalues}

Fix a group $ G \leq \GL \of{ V } $ of linear automorphisms of a finite-dimensional vector space $ V $ over $ \K $ and a subgroup $ H $ of the normalizer $ N_{ \GL \of{ V } } \of{ G } $, and consider two representations $ \rho , \eta \colon \Gamma \to G $ of $ \Gamma $ in $ G $. If $ \rho \sim_{ H } \eta $, then $ \lambda_{ k } \circ \rho = \lambda_{ k } \circ \eta $ for all $ 1 \leq k \leq \dim_{ \K } \of{ V } $.

\end{fact+}

\subsection{Eigenvalues and Anosov Representations}

The framework of Anosov representations provides a higher rank generalization of convex cocompact subgroups of rank $ 1 $ Lie groups which retains many of the most desirable dynamical properties. The Anosov condition for representations in $ \SL_{ d } \of{ \R } $ was first introduced by Labourie \cite[Section~3.1.2]{Lab06}, and has since been successfully generalized by Guichard--Wienhard \cite[Definition~2.10]{GW12} to representations in semisimple and reductive Lie groups.

Given two non-negative real-valued functions $ f , g \colon \Gamma \to \clopint{ 0 }{ \infty } $ on $ \Gamma $, we will say that $ f \of{ \gamma } $ \emph{grows at least linearly} in $ g \of{ \gamma } $ if there are constants $ a > 0 $ and $ b \geq 0 $ so that $ f \of{ \gamma } \geq a g \of{ \gamma } - b $ for all $ \gamma \in \Gamma $. This terminology is central to the fields of coarse geometry and geometric group theory, and plays a key role in multiple equivalent characterizations of Anosov representations. Specifically, an important characterization of the Anosov condition equivalent to Labourie's \cite[Section~3.1.2]{Lab06} but phrased in the above language of at least linear growth of eigenvalue magnitudes is due to Kassel--Potrie \cite[Corollary~4.6]{KP22}. In order to present this characterization, we will require the following additional terminology.

We recall that $ \Gamma $ is assumed to be finitely generated. To each finite generating set $ \Sigma $ for $ \Gamma $ is associated a \emph{translation length} function $ \norm{ \oparg }_{ \Sigma } \colon \Gamma \to \N $ defined by the formula
\[
\norm{ \gamma }_{ \Sigma } = \min \of{ \set{ n \in \N : \delta \gamma \delta^{ -1 } = \sigma_{ 1 } \dotsm \sigma_{ n } \textrm{ for some } \delta \in \Gamma \textrm{ and } \sigma_{ 1 } , \dotsc , \sigma_{ n } \in \Sigma \cup \Sigma^{ -1 } } } .
\]
Defined as such, $ \norm{ \gamma }_{ \Sigma } $ is the translation length of the action of $ \gamma $ by left multiplication on $ \Gamma $, which is assumed to be equipped with the word metric associated to $ \Sigma $.

For each finite-dimensional vector space $ V $ over a field $ \K $, either $ \R $ or $ \C $, we will denote
\[
\Delta_{ V } = \set{ 1 , 2 , \dotsc , \dim_{ \K } \of{ V } - 1 } .
\]
This set $ \Delta_{ V } $ has an \emph{opposition involution} $ \iota_{ V } \colon \Delta_{ V } \to \Delta_{ V } $ defined by the formula $ \iota_{ V } \of{ k } = \dim_{ \K } \of{ V } - k $.

\begin{definition*}[{Anosov representation, \cite[Corollary~4.6]{KP22}}]

A representation $ \rho \colon \Gamma \to \GL \of{ V } $ of $ \Gamma $ is called \emph{$ P_{ \theta } $-Anosov} for some subset $ \theta \subseteq \Delta_{ V } $ if for each $ k \in \theta $, the $ k $th \emph{logarithmic eigenvalue gap} $ \log \of{ \frac{ \lambda_{ k } \of{ \rho \of{ \gamma } } }{ \lambda_{ k + 1 } \of{ \rho \of{ \gamma } } } } $ grows at least linearly in some (equivalently, every) translation length $ \norm{ \gamma }_{ \Sigma } $; that is, $ \rho $ is $ P_{ \theta } $-Anosov if for each $ k \in \theta $ there are constants $ a_{ k } > 0 $ and $ b_{ k } \geq 0 $ so that
\[
\log \of{ \frac{ \lambda_{ k } \of{ \rho \of{ \gamma } } }{ \lambda_{ k + 1 } \of{ \rho \of{ \gamma } } } } \geq a_{ k } \norm{ \gamma }_{ \Sigma } - b_{ k }
\]
for all $ \gamma \in \Gamma $.

A representation $ \rho \colon \Gamma \to \GL \of{ V } $ is called \emph{Anosov} if it is $ P_{ \theta } $-Anosov for some non-empty subset $ \emptyset \subsetneq \theta \subseteq \Delta_{ V } $, and $ \rho $ is called \emph{Borel Anosov} if it is $ P_{ \Delta_{ V } } $-Anosov.

\end{definition*}

We have chosen to explicitly allow the possibility that $ \theta = \emptyset $ in the above definition in order to simplify several statements of results, most notably in (2) in \autoref{thm:eigenvalueConfiguration}. However, we remark that the $ P_{ \emptyset } $-Anosov condition is vacuous; that is, every representation $ \Gamma \to \GL \of{ V } $ is $ P_{ \emptyset } $-Anosov according to the above definition.

Given a finite-dimensional vector space $ V $ over $ \K $, we observe that
\[
\lambda_{ k } \of{ \vect{ A }^{ -1 } } = \frac{ 1 }{ \lambda_{ \dim_{ \K } \of{ V } - k + 1 } \of{ \vect{ A } } }
\]
for all $ 1 \leq k \leq \dim_{ \K } \of{ V } $ and $ \vect{ A } \in \GL \of{ V } $. Moreover, given a finite generating set $ \Sigma $ for $ \Gamma $, we note that $ \norm{ \gamma^{ -1 } }_{ \Sigma } = \norm{ \gamma }_{ \Sigma } $ for all $ \gamma \in \Gamma $. From these observations, one can derive the following facts:

\begin{fact+} \label{fact:AnosovInvolution}

Fix a representation $ \rho \colon \Gamma \to \GL \of{ V } $ of $ \Gamma $ and a subset $ \theta \subseteq \Delta_{ V } $.
\begin{enumerate}

\item[(1)]

$ \rho $ is $ P_{ \theta } $-Anosov if and only if it is $ P_{ \set{ k } } $-Anosov for all $ k \in \theta $.

\item[(2)]

$ \rho $ is $ P_{ \theta } $-Anosov if and only if it is $ P_{ \iota_{ V } \of{ \theta } } $-Anosov.

\end{enumerate}

\end{fact+}

One consequence of \autoref{fact:AnosovInvolution} is that for the definition of Anosov representations, we could equivalently have only considered singleton subsets of the form $ \theta = \set{ k } $ for a choice of integer $ 1 \leq k \leq \frac{ \dim_{ \K } \of{ V } }{ 2 } $. We will follow popular convention for such singletons and write $ P_{ k } $ instead of $ P_{ \set{ k } } $ both for the Anosov condition defined above and the proximality definition to come.

Note that \autoref{fact:equivalenceEigenvalues} implies that if a representation is Anosov, then so is every equivalent representation. Thus we may also refer to Anosov equivalence classes in character varieties. One of the most useful properties of the Anosov condition is that it is stable, both for representations and for their equivalence classes in the character variety.

\begin{theorem+}[{\cite[Theorem~1.2]{GW12}}] \label{thm:stability}

Fix a group $ G \leq \GL \of{ V } $ of linear automorphisms of a finite-dimensional vector space $ V $ over $ \K $ and a subgroup $ H $ of the normalizer $ N_{ \GL \of{ V } } \of{ G } $. For any subset $ \theta \subseteq \Delta_{ V } $, the set of $ P_{ \theta } $-Anosov representations of $ \Gamma $ in $ G $ is open in the representation variety $ \hom \of{ \Gamma , G } $ and descends to an open subset of the character variety $ \Char_{ H } \of{ \Gamma , G } $.

\end{theorem+}

\subsection{Flags and Proximality}

A family $ F = \of{ F_{ k } }_{ k \in \theta } $ of proper, non-zero subspaces $ \set{ \vect{ 0 } } \subsetneq F_{ k } \subsetneq V $ of a finite-dimensional vector space $ V $ over $ \K $ indexed by a subset $ \theta \subseteq \Delta_{ V } $ is a \emph{flag} in $ V $ of \emph{signature} $ \theta $ if it satisfies the following two conditions:
\begin{enumerate}

\item[(i)]

$ \dim_{ \K } \of{ F_{ k } } = k $ for all $ k \in \theta $; and

\item[(ii)]

$ F_{ k } \subsetneq F_{ \ell } $ for all $ k , \ell \in \theta $ so that $ k < \ell $.

\end{enumerate}
Such a flag $ F $ is called \emph{trivial} if $ \theta = \emptyset $, \emph{non-trivial} otherwise if $ \theta \neq \emptyset $, \emph{full} if $ \theta = \Delta_{ V } $, and \emph{partial} otherwise if $ \theta \subsetneq \Delta_{ V } $. Given a subset $ \theta \subseteq \Delta_{ V } $, the space $ \Flag_{ \theta } \of{ V } $ of flags in $ V $ of signature $ \theta $ is a smooth manifold called a \emph{flag variety}.

A linear transformation $ \vect{ A } \colon V \to V $ of a finite-dimensional vector space $ V $ is said to \emph{preserve} a flag $ F = \of{ F_{ k } }_{ k \in \theta } $ in $ V $ of signature some subset $ \theta \subseteq \Delta_{ V } $ if it preserves each subspace; that is, $ A $ preserves $ F $ if $ A \of{ F_{ k } } \subseteq F_{ k } $ for all $ k \in \theta $. More generally, if $ A \colon V \to V $ is a linear automorphism of $ V $, then $ \vect{ A } \cdot F = \of{ \vect{ A } \of{ F_{ k } } }_{ k \in \theta } $ is also a flag in $ V $ of signature $ \theta $. We thus obtain an action of the general linear group $ \GL \of{ V } $ on the flag variety $ \Flag_{ \theta } \of{ V } $ which is both smooth and transitive. The stabilizer $ \stabz_{ \GL \of{ V } } \of{ F } $ under this action consists of those linear automorphisms which preserve $ F $, and we note that
\[
\Flag_{ \theta } \of{ V } \cong \lfaktor{ \stabz \of{ F } }{ \GL \of{ V } } .
\]

A representation $ \rho \colon \Gamma \to \GL \of{ V } $ is said to \emph{preserve} a flag $ F $ in $ V $ if $ \rho \of{ \gamma } $ preserves $ F $ for each $ \gamma \in \Gamma $, or equivalently if $ \rho \of{ \Gamma } \subseteq \stabz_{ \GL \of{ V } } \of{ F } $. A representation $ \rho \colon \Gamma \to \GL \of{ V } $ is called \emph{reducible} if it preserves a non-trivial flag in $ V $ and \emph{irreducible} otherwise.

For an integer $ 0 < k \leq \dim_{ \K } \of{ V } $, the \emph{attracting $ k $-subspace} $ \vect{ A }^{ + }_{ k } $ (resp. \emph{repelling $ k $-subspace} $ \vect{ A }^{ - }_{ k } $) of a linear transformation $ \vect{ A } \colon V \to V $ is the subspace of $ V $ spanned by the generalized eigenspaces\footnote{If $ \K = \R $, then for a given eigenvalue, we instead take the intersection of $ V = V \otimes_{ \R } \R $ with the corresponding generalized eigenspace for the linear transformation $ \vect{ A } \otimes \id_{ \C } $ induced by $ \vect{ A } $ on the complexification $ V \otimes_{ \R } \C $ of $ V $.} of eigenvalues of magnitude at least $ \lambda_{ k } \of{ \vect{ A } } $ (resp. at most $ \lambda_{ \iota_{ V } \of{ k } + 1 } \of{ \vect{ A } } $). We also set $ \vect{ A }^{ + }_{ 0 } = \vect{ A }^{ - }_{ 0 } = \set{ \vect{ 0 } } $. For any such integer $ 0 \leq k \leq \dim_{ \K } \of{ V } $ and linear transformation $ \vect{ A } \colon V \to V $, the following are true:
\begin{align*}
\dim_{ \K } \of{ \vect{ A }^{ + }_{ k } } & \geq k & \dim_{ \K } \of{ \vect{ A }^{ - }_{ k } } & \geq k & V & = \lspan_{ \K } \of{ \vect{ A }^{ + }_{ k } , \vect{ A }^{ - }_{ \iota_{ V } \of{ k } } } .
\end{align*}
Moreover, these subspaces are non-decreasing in $ k $; that is,
\begin{align*}
\set{ \vect{ 0 } } = \vect{ A }^{ \pm }_{ 0 } \subsetneq \vect{ A }^{ \pm }_{ 1 } \subseteq \dotsb \subseteq \vect{ A }^{ \pm }_{ \dim_{ \K } \of{ V } - 1 } \subseteq \vect{ A }^{ \pm }_{ \dim_{ \K } \of{ V } } = V .
\end{align*}

\begin{fact+} \label{fact:proximalityCharacterizations}

Fix an integer $ k \in \Delta_{ V } $. For any linear transformation $ \vect{ A } \colon V \to V $, $ \lambda_{ k } \of{ \vect{ A } } > \lambda_{ k + 1 } \of{ \vect{ A } } $ if and only if any (and therefore all) of the following hold:
\begin{multicols}{3}
\begin{enumerate}

\item[(1)]

$ \dim_{ \K } \of{ \vect{ A }^{ + }_{ k } } = k $;

\item[(2)]

$ \dim_{ \K } \of{ \vect{ A }^{ - }_{ \iota_{ V } \of{ k } } } = \iota_{ V } \of{ k } $;

\item[(3)]

$ \vect{ A }^{ + }_{ k } \cap \vect{ A }^{ - }_{ \iota_{ V } \of{ k } } = \set{ \vect{ 0 } } $;

\item[(4)]

$ V = \vect{ A }^{ + }_{ k } \oplus \vect{ A }^{ - }_{ \iota_{ V } \of{ k } } $;

\item[(5)]

$ \vect{ A }^{ + }_{ k } \subsetneq \vect{ A }^{ + }_{ k + 1 } $; or

\item[(6)]

$ \vect{ A }^{ - }_{ \iota_{ V } \of{ k } } \subsetneq \vect{ A }^{ - }_{ \iota_{ V } \of{ k } + 1 } $.

\end{enumerate}
\end{multicols}

\end{fact+}

\begin{definition*}[Attracting/repelling flag; proximality]

A linear transformation $ \vect{ A } \colon V \to V $ of a finite-dimensional vector space $ V $ over $ \K $ is called $ P_{ \theta } $-\emph{proximal} for some subset $ \theta \subseteq \Delta_{ V } $ if $ \lambda_{ k } \of{ \vect{ A } } > \lambda_{ k + 1 } \of{ \vect{ A } } $ for all $ k \in \theta $. In this case, the finite sequences $ \vect{ A }^{ + }_{ \theta } = \of{ \vect{ A }^{ + }_{ k } }_{ k \in \theta } $ and $ \vect{ A }^{ - }_{ \iota_{ V } \of{ \theta } } = \of{ \vect{ A }^{ - }_{ k } }_{ k \in \iota \of{ \theta } } $ are transverse flags in $ V $ of signature $ \theta $ and $ \iota_{ V } \of{ \theta } $, respectively. These flags $ \vect{ A }^{ + }_{ \theta } $ and $ \vect{ A }^{ - }_{ \iota_{ V } \of{ \theta } } $ are called the \emph{attracting $ \theta $-flag} and \emph{repelling $ \iota_{ V } \of{ \theta } $-flag} of $ \vect{ A } $ because they are the unique attracting fixed point and repelling fixed point for the forward action of $ \vect{ A } $ on the flag varieties $ \Flag_{ \theta } \of{ V } $ and $ \Flag_{ \iota_{ V } \of{ \theta } } \of{ V } $, respectively.

\end{definition*}

\begin{fact+} \label{fact:biproximality}

For a linear automorphism $ \vect{ A } \colon V \to V $ of a finite-dimensional vector space $ V $ over $ \K $ and a subset $ \theta \subseteq \Delta_{ V } $, the following are equivalent:
\begin{multicols}{2}
\begin{enumerate}

\item[(1)]

$ \vect{ A } $ is both $ P_{ \theta } $- and $ P_{ \iota_{ V } \of{ \theta } } $-proximal;

\item[(3)]

$ \vect{ A } $ and $ \vect{ A }^{ -1 } $ are $ P_{ \theta } $-proximal; and

\item[(2)]

$ \vect{ A }^{ -1 } $ is both $ P_{ \theta } $- and $ P_{ \iota_{ V } \of{ \theta } } $-proximal;

\item[(4)]

$ \vect{ A } $ and $ \vect{ A }^{ -1 } $ are $ P_{ \iota_{ V } \of{ \theta } } $-proximal.

\end{enumerate}
\end{multicols}
A linear automorphism $ \vect{ A } \colon V \to V $ which satisfies any (and therefore all) of the above conditions is called $ P_{ \theta } $-\emph{biproximal}. In this case,
\begin{align*}
\of{ \vect{ A }^{ -1 } }^{ + }_{ \theta } & = \vect{ A }^{ - }_{ \theta } & \of{ \vect{ A }^{ -1 } }^{ - }_{ \theta } & = \vect{ A }^{ + }_{ \theta } & \of{ \vect{ A }^{ -1 } }^{ + }_{ \iota_{ V } \of{ \theta } } & = \vect{ A }^{ - }_{ \iota_{ V } \of{ \theta } } & \of{ \vect{ A }^{ -1 } }^{ - }_{ \iota_{ V } \of{ \theta } } & = \vect{ A }^{ + }_{ \iota_{ V } \of{ \theta } } .
\end{align*}

\end{fact+}

A representation $ \rho \colon \Gamma \to \GL \of{ V } $ is called $ P_{ \theta } $-\emph{proximal} if the image $ \rho \of{ \gamma } $ of any infinite order $ \gamma \in \Gamma $ is $ P_{ \theta } $-proximal. In particular, since this also holds for $ \gamma^{ -1 } $, any $ P_{ \theta } $-proximal representation has the stronger property that the image of any infinite order element is $ P_{ \theta } $-biproximal. Anosov representations are our primary example of proximal representations. In fact, something much stronger is true: the closure of the set of attracting flags of the images of infinite order elements under an Anosov representation of $ \Gamma $ is homeomorphic to its Gromov boundary $ \partial \Gamma $.

\begin{theorem+}[{\cite[Corollary~4.6]{KP22}}]

Fix a representation $ \rho \colon \Gamma \to \GL \of{ V } $ of $ \Gamma $ and a subset $ \theta \subseteq \Delta_{ V } $. If $ \rho $ is $ P_{ \theta } $-Anosov, then there are continuous and $ \rho $-equivariant maps $ \xi_{ \rho }^{ \theta } \colon \partial \Gamma \to \Flag_{ \theta } \of{ V } $ and $ \xi_{ \rho }^{ \iota_{ V } \of{ \theta } } \colon \partial \Gamma \to \Flag_{ \iota_{ V } \of{ \theta } } \of{ V } $ defined on the Gromov boundary $ \partial \Gamma $ with the following properties:
\begin{enumerate}

\item[(1)]

Transversality: For all $ x, y \in \partial \Gamma $, if $ x \neq y $, then $ \xi_{ \rho }^{ \theta } \of{ x } $ and $ \xi_{ \rho }^{ \iota_{ V } \of{ \theta } } \of{ y } $ are transverse; that is, if $ \xi_{ \rho }^{ \theta } \of{ x } = \of{ F_{ k } }_{ k \in \theta } $ and $ \xi_{ \rho }^{ \iota \of{ \theta } } = \of{ G_{ \ell } }_{ \ell \in \iota_{ V } \of{ \theta } } $, then $ V = F_{ k } \oplus G_{ \iota_{ V } \of{ k } } $ for all $ k \in \theta $.

\item[(2)]

$ P_{ \theta } $-dynamics preservation: $ \rho $ is $ P_{ \theta } $-proximal, and for all infinite order $ \gamma \in \Gamma $,
\begin{align*}
\xi_{ \rho }^{ \theta } \of{ \gamma^{ + } } & = \rho \of{ \gamma }^{ + }_{ \theta } & \xi_{ \rho }^{ \theta } \of{ \gamma^{ - } } & = \rho \of{ \gamma }^{ - }_{ \theta } & \xi_{ \rho }^{ \iota_{ V } \of{ \theta } } \of{ \gamma^{ + } } & = \rho \of{ \gamma }^{ + }_{ \iota_{ V } \of{ \theta } } & \xi_{ \rho }^{ \iota_{ V } \of{ \theta } } \of{ \gamma^{ - } } & = \rho \of{ \gamma }^{ - }_{ \iota_{ V } \of{ \theta } } ,
\end{align*}
where $ \gamma^{ \pm } \in \partial \Gamma $ are the attracting and repelling ideal fixed points of $ \gamma $ in $ \partial \Gamma $.

\end{enumerate}
These maps are the unique continuous and $ \rho $-equivariant maps with the above properties, and are called the \emph{$ P_{ \theta } $-Anosov limit maps} of $ \rho $.

\end{theorem+}

\subsection{Direct Sum Decompositions and Block Diagonal Representations}

A finite sequence $ \mathscr{ U } = \of{ U_{ j } }_{ j = 1 }^{ m } $ of subspaces $ U_{ j } \subseteq V $, called \emph{factors}, of a finite-dimensional vector space $ V $ over $ \K $ is called a \emph{direct sum decomposition} of $ V $ if $ V = \lspan_{ \K } \of{ U_{ 1 } , \dotsc , U_{ m } } $ and $ U_{ i } \cap U_{ j } = \set{ \vect{ 0 } } $ for all $ 0 \leq i < j \leq m $. In this case, we write $ V = \bigoplus_{ j = 1 }^{ m }{ U_{ j } } $ and note that $ \dim_{ \K } \of{ V } = \sum_{ j = 1 }^{ m }{ \dim_{ \K } \of{ U_{ j } } } $. Each direct sum decomposition $ \mathscr{ U } = \of{ U_{ j } }_{ j = 1 }^{ m } $ of $ V $ has an associated flag in $ V $ of signature
\[
\theta = \set{ \sum_{ i = 1 }^{ j }{ \dim_{ \K } \of{ U_{ i } } } : 1 \leq j \leq m } \cap \Delta_{ V } .
\]
The subspaces of this associated flag are defined by the partial sums $ \bigoplus_{ i = 1 }^{ j_{ k } }{ U_{ i } } $, where for each $ k \in \theta $, $ 1 \leq j_{ k } \leq m $ is any integer so that $ \sum_{ i = 1 }^{ j_{ k } }{ \dim_{ \K } \of{ U_{ i } } } = k $. We note that every non-trivial flag in $ V $ is the associated flag of several (in fact, uncountably infinitely many) direct sum decompositions of $ V $.

A linear transformation $ \vect{ A } \colon V \to V $ is said to \emph{preserve} a direct sum decomposition $ \mathscr{ U } = \of{ U_{ j } }_{ j = 1 }^{ m } $ of $ V $ if it preserves each factor; that is, if $ \vect{ A } \of{ U_{ j } } \subseteq U_{ j } $ for all $ 1 \leq j \leq m $. From this condition we obtain the subsets
\begin{align*}
\End_{ \K } \of{ \mathscr{ U } } & = \set{ \vect{ A } \in \End_{ \K } \of{ V } : \vect{ A } \of{ U_{ j } } \subseteq U_{ j } \textrm{ for all } 1 \leq j \leq m } , \\
\Aut_{ \K } \of{ \mathscr{ U } } & = \set{ \vect{ A } \in \GL \of{ V } : \vect{ A } \of{ U_{ j } } = U_{ j } \textrm{ for all } 1 \leq j \leq m } = \End_{ \K } \of{ \mathscr{ U } } \cap \GL \of{ V } .
\end{align*}
Note that any linear transformation which preserves a direct sum decomposition also preserves its associated flag; that is, block diagonal linear transformations are block upper triangular. However, the converse to this statement is false as long as said flag is non-trivial.

\begin{fact+}

A linear transformation $ \vect{ A } \colon V \to V $ of a finite-dimensional vector space $ V $ over $ \K $ is block diagonal relative to a direct sum decomposition $ \mathscr{ U } = \of{ U_{ j } }_{ j = 1 }^{ m } $ of $ V $ if and only if there is a sequence $ \of{ \vect{ A }_{ j } }_{ j = 1 }^{ m } $ of linear transformations $ \vect{ A }_{ j } \colon U_{ j } \to U_{ j } $ so that
\[
\vect{ A } \of{ \vect{ v } } = \sum_{ j = 1 }^{ m }{ \vect{ A }_{ j } \of{ \pi_{ j } \of{ \vect{ v } } } }
\]
for all $ \vect{ v } \in V $, where $ \pi_{ j } \colon V \onto U_{ j } $ is projection onto the $ j $th factor $ U_{ j } $ with kernel spanned by the other factors. Such a sequence $ \of{ \vect{ A }_{ j } }_{ j = 1 }^{ m } $ is uniquely defined by the above property and is called the \emph{block decomposition} of $ \vect{ A } $ relative to $ \mathscr{ U } $. $ \vect{ A }_{ j } $ is called the $ j $th \emph{block} of $ \vect{ A } $ relative to $ \mathscr{ U } $, and we write $ \vect{ A } = \bigoplus_{ j = 1 }^{ m }{ \vect{ A }_{ j } } $, so that
\begin{align*}
\End_{ \K } \of{ \mathscr{ U } } & = \set{ \bigoplus_{ j = 1 }^{ m }{ \vect{ A }_{ j } } \in \End_{ \K } \of{ V } : \vect{ A }_{ j } \in \End_{ \K } \of{ U_{ j } } \textrm{ for all } 1 \leq j \leq m } = \bigoplus_{ j = 1 }^{ m }{ \End_{ \K } \of{ U_{ j } } } ,
\end{align*}
and similarly
\begin{align*}
\Aut_{ \K } \of{ \mathscr{ U } } & = \set{ \bigoplus_{ j = 1 }^{ m }{ \vect{ A }_{ j } } \in \GL \of{ V } : \vect{ A }_{ j } \in \GL \of{ U_{ j } } \textrm{ for all } 1 \leq j \leq m } \cong \prod_{ j = 1 }^{ m }{ \GL \of{ U_{ j } } } .
\end{align*}

\end{fact+}

Analogously, a representation $ \rho \colon \Gamma \to \GL \of{ V } $ is said to \emph{preserve} a direct sum decomposition $ \mathscr{ U } $ of $ V $ if $ \rho \of{ \gamma } $ preserves $ \mathscr{ U } $ for each $ \gamma \in \Gamma $; that is, $ \rho $ preserves $ \mathscr{ U } $ if $ \rho \of{ \Gamma } \subseteq \Aut_{ \K } \of{ \mathscr{ U } } $. In this case, the map $ \rho_{ j } \colon \Gamma \to \GL \of{ U_{ j } } $ which on input $ \gamma \in \Gamma $ returns the $ j $th block of $ \rho \of{ \gamma } $ relative to $ \mathscr{ U } $ is a representation of $ \Gamma $ in $ U_{ j } $ called the $ j $th \emph{block} of $ \rho $ relative to $ \mathscr{ U } $. $ \rho $ is called \emph{semisimple} or \emph{completely reducible} if it preserves a direct sum decomposition so that each block is irreducible. In this case, $ \rho $ is said to be \emph{completely reduced} relative to said direct sum decomposition.

We now describe a standard form for representations in $ V $ which preserve a given direct sum decomposition $ \mathscr{ U } $ of $ V $. This standard form derives from a direct product decomposition of $ \Aut_{ \K } \of{ \mathscr{ U } } $.

\begin{proposition+} \label{prop:blockDiagonalIsomorphism}

Fix a direct sum decomposition $ \mathscr{ U } = \of{ U_{ j } }_{ j = 1 }^{ m } $ of a finite-dimensional non-zero vector space $ V $ over $ \K $, and consider the subspace
\[
D_{ \mathscr{ U } } = \set{ \vect{ x } \in \R^{ m } : \sum_{ j = 1 }^{ m }{ \dim_{ \K } \of{ U_{ j } } x_{ j } } = 0 , \textrm{ and } x_{ j } = 0 \textrm{ for all } 1 \leq j \leq m \textrm{ so that } U_{ j } = \set{ \vect{ 0 } } }
\]
and the group
\[
N_{ \mathscr{ U } } = \set{ \bigoplus_{ j = 1 }^{ m }{ \vect{ B }_{ j } } \in \Aut_{ \K } \of{ \mathscr{ U } } : \vect{ B }_{ j } \in \SL^{ * } \of{ U_{ j } } \textrm{ for all } 1 \leq j \leq m } \cong \prod_{ j = 1 }^{ m }{ \SL^{ * } \of{ U_{ j } } } .
\]
There is a bijection $ \beta_{ \mathscr{ U } } \colon \hom \of{ \Gamma , \R } \times \hom \of{ \Gamma , D_{ \mathscr{ U } } } \times \hom \of{ \Gamma , N_{ \mathscr{ U } } } \to \hom \of{ \Gamma , \Aut_{ \K } \of{ \mathscr{ U } } } $ which induces a bijection $ \beta_{ \mathscr{ U } } \of{ 0 , \oparg , \oparg } \colon \hom \of{ \Gamma , D_{ \mathscr{ U } } } \times \hom \of{ \Gamma , N_{ \mathscr{ U } } } \to \hom \of{ \Gamma , \Aut_{ \K } \of{ \mathscr{ U } } \cap \SL^{ * } \of{ V } } $.

\begin{proof}

One useful map $ \Psi_{ \mathscr{ U } } \colon \R \times D_{ \mathscr{ U } } \times N_{ \mathscr{ U } } \to \Aut_{ \K } \of{ \mathscr{ U } } $ is given by the formula
\[
\Psi_{ \mathscr{ U } } \of{ s , \vect{ x } , \bigoplus_{ j = 1 }^{ m }{ \vect{ B }_{ j } } } = \bigoplus_{ j = 1 }^{ m }{ e^{ s + x_{ j } } \vect{ B }_{ j } } .
\]
One can check that this defines a group isomorphism, with inverse given by the formula
\[
\Psi_{ \mathscr{ U } }^{ -1 } \of{ \bigoplus_{ j = 1 }^{ m }{ \vect{ A }_{ j } } } = \of{ \frac{ \log \of{ \abs{ \det \of{ \vect{ A } } } } }{ \dim_{ \K } \of{ V } } , \of{ \frac{ \log \of{ \abs{ \det \of{ \vect{ A }_{ j.
} } } } }{ \dim_{ \K } \of{ U_{ j } } } - \frac{ \log \of{ \abs{ \det \of{ \vect{ A } } } } }{ \dim_{ \K } \of{ V } } }_{ j = 1 }^{ m } , \bigoplus_{ j = 1 }^{ m }{ \frac{ \vect{ A }_{ j } }{ \abs{ \det \of{ \vect{ A }_{ j } } }^{ \frac{ 1 }{ \dim_{ \K } \of{ U_{ j } } } } } } } .
\]
We note that if the $ j $th factor $ U_{ j } $ is $ \set{ \vect{ 0 } } $ for some $ 1 \leq j \leq m $, then parts of the above formula are undefined. Specifically, the $ j $th component of the resulting vector in $ D_{ \mathscr{ U } } $ and the $ j $th block of the linear transformation in $ N_{ \mathscr{ U } } $ are not defined. In this case, the definition of $ D_{ \mathscr{ U } } $ and the unique choice of linear transformation $ U_{ j } \to U_{ j } $ should make clear what should replace these formulas, respectively. We will allow this abuse of notation for ease of legibility.

Given homomorphisms $ \delta \colon \Gamma \to \R $ and $ \varphi \colon \Gamma \to D_{ \mathscr{ U } } $ and a representation $ \zeta \colon \Gamma \to N_{ \mathscr{ U } } $, we define a map $ \eta = \beta_{ \mathscr{ U } } \of{ \delta , \varphi , \zeta } \colon \Gamma \to \Aut_{ \K } \of{ \mathscr{ U } } $ by the formula $ \eta \of{ \gamma } = \Psi_{ \mathscr{ U } } \of{ \delta \of{ \gamma } , \varphi \of{ \gamma } , \zeta \of{ \gamma } } $. Note that since $ \Psi_{ \mathscr{ U } } $, $ \delta $, $ \varphi $, and $ \zeta $ are homomorphisms, $ \eta \in \hom \of{ \Gamma , \Aut_{ \K } \of{ \mathscr{ U } } } $. That $ \beta_{ \mathscr{ U } } \colon \hom \of{ \Gamma , \R } \times \hom \of{ \Gamma , D_{ \mathscr{ U } } } \times \hom \of{ \Gamma , N_{ \mathscr{ U } } } \to \hom \of{ \Gamma , \Aut_{ \K } \of{ \mathscr{ U } } } $ is a bijection is a standard consequence of the properties of direct products of groups.

We also note that
\begin{align*}
\abs{ \det \of{ \Psi_{ \mathscr{ U } } \of{ s , \vect{ x } , \bigoplus_{ j = 1 }^{ m }{ \vect{ B }_{ j } } } } } & = \abs{ \det \of{ \bigoplus_{ j = 1 }^{ m }{ e^{ s + x_{ j } } \vect{ B }_{ j } } } } = \prod_{ j = 1 }^{ m }{ \abs{ \det \of{ e^{ s + x_{ j } } \vect{ B }_{ j } } } } \\
& = \prod_{ j = 1 }^{ m }{ e^{ \dim_{ \K } \of{ U_{ j } } \of{ s + x_{ j } } } \abs{ \det \of{ \vect{ B }_{ j } } } } = \exp \of{ \sum_{ j = 1 }^{ m }{ \dim_{ \K } \of{ U_{ j } } \of{ s + x_{ j } } } } \\
& = e^{ \dim_{ \K } \of{ V } s + 0 } = e^{ \dim_{ \K } \of{ V } s }
\end{align*}
for all $ s \in \R $, $ \vect{ x } \in D_{ \mathscr{ U } } $, and $ \bigoplus_{ j = 1 }^{ m }{ \vect{ B }_{ j } } \in N_{ \mathscr{ U } } $, so that $ \Psi_{ \mathscr{ U } } \of{ s , \vect{ x } , \bigoplus_{ j = 1 }^{ m }{ \vect{ B }_{ j } } } \in \SL^{ * } \of{ V } $ if and only if $ s = 0 $. Thus $ \Psi_{ \mathscr{ U } } \colon \R \times D_{ \mathscr{ U } } \times N_{ \mathscr{ U } } \to \Aut_{ \K } \of{ V } $ induces an isomorphism $ \Psi_{ \mathscr{ U } } \of{ 0 , \oparg , \oparg } \colon D_{ \mathscr{ U } } \times N_{ \mathscr{ U } } \to \Aut_{ \K } \of{ \mathscr{ U } } \cap \SL^{ * } \of{ V } $, and so $ \beta_{ \mathscr{ U } } \of{ 0 , \oparg , \oparg } \colon \hom \of{ \Gamma , D_{ \mathscr{ U } } } \times \hom \of{ \Gamma , N_{ \mathscr{ U } } } \to \hom \of{ \Gamma , \Aut_{ \K } \of{ \mathscr{ U } } \cap \SL^{ * } \of{ V } } $ is a bijection. \qedhere

\end{proof}

\end{proposition+}

For a direct sum decomposition $ \mathscr{ U } = \of{ U_{ j } }_{ j = 1 }^{ m } $ of a non-zero finite-dimensional vector space $ V $ over $ \K $, given homomorphisms $ \delta \colon \Gamma \to \R $ and $ \varphi \colon \Gamma \to D_{ \mathscr{ U } } $ and a representation $ \zeta \colon \Gamma \to N_{ \mathscr{ U } } $, we will call the representation $ \eta = \beta_{ \mathscr{ U } } \of{ \delta , \varphi , \zeta } \colon \Gamma \to \Aut_{ \K } \of{ \mathscr{ U } } $ defined above in \autoref{prop:blockDiagonalIsomorphism} the associated \emph{block diagonal representation}. In this setting, we will refer to $ \varphi $ and $ \zeta $ as the \emph{block deformation} and \emph{block normalization}, respectively, of $ \eta $ relative to $ \mathscr{ U } $; representations with image contained in $ N_{ \mathscr{ U } } $ will be called \emph{block normalized} relative to $ \mathscr{ U } $.

For a concrete formula for $ \beta_{ \mathscr{ U } } \of{ \delta , \varphi , \zeta } $, let $ \varphi_{ j } \colon \Gamma \to \R $ and $ \zeta_{ j } \colon \Gamma \to \SL^{ * } \of{ U_{ j } } $ denote the $ j $th component of the block deformation $ \varphi \colon \Gamma \to D_{ \mathscr{ U } } $ and the $ j $th block of the block normalization $ \zeta \colon \Gamma \to N_{ \mathscr{ U } } $, respectively, so that $ \varphi \of{ \gamma } = \of{ \varphi_{ j } \of{ \gamma } }_{ j = 1 }^{ m } $ and $ \zeta \of{ \gamma } = \bigoplus_{ j = 1 }^{ m }{ \zeta_{ j } \of{ \gamma } } $ for all $ \gamma \in \Gamma $. Then
\[
\eta \of{ \gamma } = \bigoplus_{ j = 1 }^{ m }{ e^{ \delta \of{ \gamma } + \varphi_{ j } \of{ \gamma } } \zeta_{ j } \of{ \gamma } }
\]
for all $ \gamma \in \Gamma $.

We remark that while the above formula for $ \beta_{ \mathscr{ U } } \of{ \delta , \varphi , \zeta } $ is a generalization of that of the semisimplifications of the reducible suspensions introduced in \cite[Section~1]{Lah24}, which themselves generalize the linear $ \vect{ u } $-deformations described by Barbot \cite[Section~4.1.2]{Bar10}, every linear representation of $ \Gamma $ is a block diagonal representation (choose for example $ \mathscr{ U } $ to be the trivial direct sum decomposition). In fact, every such linear representation is equivalent to a semisimple block diagonal representation. 

\begin{corollary+} \label{cor:blockDiagonalization}

Every representation $ \rho \colon \Gamma \to \GL \of{ V } $ of $ \Gamma $ is $ \SL \of{ V } $-equivalent to a block diagonal representation $ \eta = \beta_{ \mathscr{ U } } \of{ \delta , \varphi , \zeta } $. Moreover,
\begin{enumerate}

\item[(1)]

We can choose the direct sum decomposition $ \mathscr{ U } $ so that the block diagonalization $ \eta $ and its block normalization $ \zeta $ relative to $ \mathscr{ U } $ are completely reduced relative to $ \mathscr{ U } $; and

\item[(2)]

For any subset $ \theta \subseteq \Delta_{ V } $, one of $ \rho $, $ \eta $, and $ \beta_{ \mathscr{ U } } \of{ 0 , \varphi , \zeta } $ is $ P_{ \theta } $-Anosov if and only if all are.

\end{enumerate}

\begin{proof}

Let $ \mathscr{ U } = \of{ U_{ j } }_{ j = 1 }^{ m } $ be a direct sum decomposition of $ V $ with associated flag $ F $, and consider the linear transformation $ \vect{ C }_{ \mathscr{ U } } \colon V \to V $ and the map $ B_{ \mathscr{ U } } \colon \stabz_{ \GL \of{ V } } \of{ F } \to \End_{ \K } \of{ \mathscr{ U } } $ defined by the formulas
\begin{align*}
\vect{ C }_{ \mathscr{ U } } & = \bigoplus_{ j = 1 }^{ m }{ e^{ j - \floor{ \frac{ m }{ 2 } } - 1 } \id_{ U_{ j } } } & B_{ \mathscr{ U } } \of{ \vect{ A } } & = \lim_{ n \to \infty }{ \vect{ C }_{ \mathscr{ U } }^{ n } \vect{ A } \vect{ C }_{ \mathscr{ U } }^{ -n } } = \bigoplus_{ j = 1 }^{ m }{ \pi_{ j } \circ \vect{ A } \circ \iota_{ j } } ,
\end{align*}
where for each $ 1 \leq j \leq m $, $ \iota_{ j } \colon U_{ j } \into V $ and $ \pi_{ j } \colon V \onto U_{ j } $ are inclusion of and projection onto the $ j $th factor $ U_{ j } $. One can check that the above limit always converges to the stated formula, and that said formula defines a homomorphism of $ \K $-algebras. In particular, $ B_{ \mathscr{ U } } $ restricts to a group homomorphism $ \Aut_{ \K } \of{ F } \to \Aut_{ \K } \of{ \mathscr{ U } } $.

We observe that if $ \rho $ preserves $ F $, then the composition $ \eta \coloneqq B_{ \mathscr{ U } } \circ \rho \colon \Gamma \to \Aut_{ \K } \of{ \mathscr{ U } } $ is $ \SL \of{ V } $-equivalent to $ \rho $, since $ \vect{ C }_{ \mathscr{ U } } \in \SL \of{ V } $. Moreover, since $ \eta $ preserves $ \mathscr{ U } $, $ \eta = \beta_{ \mathscr{ U } } \of{ \delta , \varphi , \zeta } $ for some homomorphisms $ \delta \colon \Gamma \to \R $ and $ \varphi \colon \Gamma \to D_{ \mathscr{ U } } $ and some block normalized representation $ \zeta \colon \Gamma \to N_{ \mathscr{ U } } $.

\begin{enumerate}

\item[(1)]

If we choose $ \mathscr{ U } $ so that $ F $ is maximal among flags in $ V $ preserved by $ \rho $, then $ \eta $ and $ \zeta $ are completely reduced relative to $ \mathscr{ U } $. This is because any non-trivial flag in a factor of $ \mathscr{ U } $ preserved by $ \eta $ or $ \zeta $ can be used to extend $ F $ to a flag of strictly larger signature preserved by $ \rho $.

\item[(2)]

Since $ \rho $ and $ \eta $ are $ \SL \of{ V } $-equivalent, one is $ P_{ \theta } $-Anosov if and only if the other is. Since
\[
\eta \of{ \gamma } = \beta_{ \mathscr{ U } } \of{ \delta , \varphi , \zeta } \of{ \gamma } = \bigoplus_{ j = 1 }^{ m }{ e^{ \delta \of{ \gamma } + \varphi_{ j } \of{ \gamma } } \zeta_{ j } \of{ \gamma } } = e^{ \delta \of{ \gamma } } \bigoplus_{ j = 1 }^{ m }{ e^{ 0 + \varphi_{ j } \of{ \gamma } } \zeta_{ j } \of{ \gamma } } = e^{ \delta \of{ \gamma } } \beta_{ \mathscr{ U } } \of{ 0 , \varphi , \zeta } \of{ \gamma }
\]
for all $ \gamma \in \Gamma $, $ \eta $ and $ \beta_{ \mathscr{ U } } \of{ 0 , \varphi , \zeta } $ have the same logarithmic eigenvalue gaps. Specifically,
\[
\log \of{ \frac{ \lambda_{ k } \of{ \eta \of{ \gamma } } }{ \lambda_{ k + 1 } \of{ \eta \of{ \gamma } } } } = \log \of{ \frac{ e^{ \delta \of{ \gamma } } \lambda_{ k } \of{ \beta_{ \mathscr{ U } } \of{ 0 , \varphi , \zeta } \of{ \gamma } } }{ e^{ \delta \of{ \gamma } } \lambda_{ k + 1 } \of{ \beta_{ \mathscr{ U } } \of{ 0 , \varphi , \zeta } \of{ \gamma } } } } = \log \of{ \frac{ \lambda_{ k } \of{ \beta_{ \mathscr{ U } } \of{ 0 , \varphi , \zeta } \of{ \gamma } } }{ \lambda_{ k + 1 } \of{ \beta_{ \mathscr{ U } } \of{ 0 , \varphi , \zeta } \of{ \gamma } } } }
\]
for all $ k \in \theta $ and $ \gamma \in \Gamma $, and therefore $ \eta $ is $ P_{ \theta } $-Anosov if and only if $ \beta_{ \mathscr{ U } } \of{ 0 , \varphi , \delta } $ is. \qedhere

\end{enumerate}

\end{proof}

\end{corollary+}

The homomorphism $ B_{ \mathscr{ U } } \colon \stabz_{ \GL \of{ V } } \of{ F } \to \End_{ \K } \of{ \mathscr{ U } } $ obtained above in the proof of \autoref{cor:blockDiagonalization} is called \emph{block diagonalization} relative to $ \mathscr{ U } $. If a representation or linear transformation obtained via this method is completely reduced relative to $ \mathscr{ U } $, then such a block diagonalization is called the \emph{semisimplification} relative to $ \mathscr{ U } $. The dynamical properties of such semisimplifications were investigated by Guichard--Gueritaud--Kassel--Wienhard \cite[Section~2.5.4, Proposition~1.8]{GGKW17}, who proved before the introduction of Kassel--Potrie's \cite[Corollary~4.6]{KP22} characterization of Anosov representations that a representation is Anosov if and only if its semisimplification is.

\section{The Anosov Condition for Reducible Representations} \label{sec:AnosovCondition}

\autoref{cor:blockDiagonalization} implies that in order to understand which reducible representations of $ \Gamma $ are Anosov, it suffices to develop such an understanding for block diagonal representations in terms of conditions on the blocks. Toward this end, fix for this section a direct sum decomposition $ \mathscr{ U } = \of{ U_{ j } }_{ j = 1 }^{ m } $ of a non-zero finite-dimensional vector space $ V $ over $ \K $ with projection maps $ \pi_{ j } \colon V \onto U_{ j } $ and partial sums
\[
U_{ j }' = \bigoplus_{ i = 1 }^{ j }{ U_{ i } } ,
\]
for all $ 1 \leq j \leq m $. We recall that all but the last of these partial sums are precisely the subspaces comprising the flag associated to $ \mathscr{ U } $.

\subsection{Eigenvalues of Block Upper Triangular Linear Transformations}

Towards our goal of adapting the eigenvalue characterization of Anosov representations introduced by Kassel--Potrie \cite{KP22} to reducible representations, we will first consider the attracting and repelling subspaces and eigenvalue magnitudes of block upper triangular linear transformations and their block diagonalizations.

\begin{proposition+} \label{prop:reducibleEigenvalues}

Fix a linear transformation $ \vect{ A } \colon V \to V $ which is block upper triangular relative to $ \mathscr{ U } $, and consider its block diagonalization $ \vect{ B } = \bigoplus_{ j = 1 }^{ m }{ \vect{ B }_{ j } } $, as described in \autoref{cor:blockDiagonalization}. For each $ 1 \leq j \leq m $, denote by $ \vect{ A }_{ j }' \colon U_{ j }' \to U_{ j }' $ the restriction of $ \vect{ A } $ to $ U_{ j }' $. Then the following hold:
\begin{enumerate}

\item[(1)]

$ U_{ j }' \cap \vect{ A }^{ \pm }_{ k } = \of{ \vect{ A }_{ j }' }^{ \pm }_{ \dim_{ \K } \of{ U_{ j }' \cap \vect{ A }^{ \pm }_{ k } } } $ and $ \pi_{ j } \of{ U_{ j }' \cap \vect{ A }^{ \pm }_{ k } } = U_{ j } \cap \vect{ B }^{ \pm }_{ k } = \of{ \vect{ B }_{ j } }^{ \pm }_{ \dim_{ \K } \of{ U_{ j } \cap \vect{ B }^{ \pm }_{ k } } } $ for all $ 1 \leq j \leq m $ and $ 0 \leq k \leq \dim_{ \K } \of{ V } $.

\item[(2)]

$ \vect{ A }^{ \pm }_{ k } \cong \bigoplus_{ j = 1 }^{ m }{ \pi_{ j } \of{ U_{ j }' \cap \vect{ A }^{ \pm }_{ k } } } = \bigoplus_{ j = 1 }^{ m }{ U_{ j } \cap \vect{ B }^{ \pm }_{ k } } = \vect{ B }^{ \pm }_{ k } $ for all $ 0 \leq k \leq \dim_{ \K } \of{ V } $.

\item[(3)]

$ U_{ j }' = \lspan_{ \K } \of{ U_{ j }' \cap \vect{ A }^{ + }_{ k } , U_{ j }' \cap \vect{ A }^{ - }_{ \dim_{ \K } \of{ V } - k } } $ and $ U_{ j } = \lspan_{ \K } \of{ U_{ j } \cap \vect{ B }^{ + }_{ k } , U_{ j } \cap \vect{ B }^{ - }_{ \dim_{ \K } \of{ V } - k } } $ for all $ 1 \leq j \leq m $ and $ 0 \leq k \leq \dim_{ \K } \of{ V } $.

\item[(4)]

$ \vect{ A } $ and $ \vect{ B } $ have eigenvalue magnitudes
\begin{align*}
\lambda_{ k } \of{ \vect{ A } } & = \lambda_{ k } \of{ \vect{ B } } = \min_{ i }{ \lambda_{ \dim_{ \K } \of{ \pi_{ i } \of{ U_{ i }' \cap \vect{ A }^{ + }_{ k } } } } \of{ \vect{ B }_{ i } } } = \max_{ j }{ \lambda_{ \dim_{ \K } \of{ U_{ j } } - \dim_{ \K } \of{ \pi_{ j } \of{ U_{ j }' \cap \vect{ A }^{ - }_{ \iota_{ V } \of{ k } + 1 } } + 1 } } \of{ \vect{ B }_{ j } } }
\end{align*}
for all $ 1 \leq k \leq \dim_{ \K } \of{ V } $, where the above minimum and maximum are index by all $ 1 \leq i , j \leq m $ so that $ \dim_{ \K } \of{ \pi_{ i } \of{ U_{ i }' \cap \vect{ A }^{ + }_{ k } } } > 0 $ and $ \dim_{ \K } \of{ \pi_{ j } \of{ U_{ j }' \cap \vect{ A }^{ - }_{ \iota_{ V } \of{ k } + 1 } } } > 0 $.

\item[(5)]

$ \vect{ B }_{ j } $ is $ P_{ \theta_{ j } } $-biproximal for all $ 1 \leq j \leq m $, where $ \theta_{ j } = \set{ \dim_{ \K } \of{ \pi_{ j } \of{ U_{ j }' \cap \vect{ A }^{ + }_{ k } } } : k \in \Delta_{ V } } \cap \Delta_{ U_{ j } } $.

\end{enumerate}

\begin{proof}

We will prove these results under the assumption that $ \K = \C $. If $ \K = \R $, then the same results may be proved by applying the following arguments to the linear transformation $ \vect{ A } \otimes \id_{ \C } $ induced by $ \vect{ A } $ on the complexification $ V \otimes_{ \R } \C $ of $ V $.

We begin by recalling from the proof of \autoref{cor:blockDiagonalization} that $ \vect{ B } = \lim_{ n \to \infty }{ \vect{ C }_{ \mathscr{ U } }^{ n } \vect{ A } \vect{ C }_{ \mathscr{ U } }^{ -n } } $, where
\[
\vect{ C }_{ \mathscr{ U } } = \bigoplus_{ j = 1 }^{ m }{ e^{ j - \floor{ \frac{ m }{ 2 } } - 1 } \id_{ U_{ j } } } .
\]
This implies that $ \lambda_{ k } \of{ \vect{ A } } = \lambda_{ k } \of{ \vect{ B } } $ for all $ 1 \leq k \leq \dim_{ \K } \of{ V } $. Inspecting the above formula for $ \vect{ C }_{ \mathscr{ U } } $, we also note that
\[
\lim_{ n \to \infty }{ e^{ -\of{ j - \floor{ \frac{ m }{ 2 } } - 1 } n } \vect{ C }_{ \mathscr{ U } }^{ n } \of{ \vect{ v } } } = \pi_{ j } \of{ \vect{ v } }
\]
for all $ 1 \leq j \leq m $ and $ \vect{ v } \in U_{ j }' $.

Proceeding inductively, we fix pairwise disjoint sets $ E_{ 1 } , \dotsc , E_{ m } \subsetneq V $ of generalized eigenvectors of $ \vect{ A } $ so that for each $ 1 \leq j \leq m $, $ E_{ j }' = \bigsqcup_{ i = 1 }^{ j }{ E_{ i } } $ is a basis for $ U_{ j }' = \bigoplus_{ i = 1 }^{ j }{ U_{ i } } $. This inductive construction ensures that $ \card{ E_{ j } } = \dim_{ \K } \of{ U_{ j } } $, and we will label these generalized eigenvectors by
\[
E_{ j } = \set{ \vect{ e }_{ j , \ell } \in U_{ j }' : 1 \leq \ell \leq \dim_{ \K } \of{ U_{ j } } } .
\]
We note that $ \pi_{ j } \colon U_{ j }' \onto U_{ j } $ is surjective, and $ \bigsqcup_{ i = 1 }^{ j - 1 }{ E_{ i } } \subsetneq \ker \of{ \pi_{ j } } $, so that $ \pi_{ j } \of{ E_{ j } } $ spans $ U_{ j } $. We conclude by the rank-nullity theorem that $ \pi_{ j } $ restricts to an isomorphism $ \lspan_{ \K } \of{ E_{ j } } \to U_{ j } $, so that $ \pi_{ j } \of{ E_{ j } } $ is a basis for $ U_{ j } $.

Denote by $ \mu_{ j , \ell } \in \K^{ \times } $ and $ 1 \leq r_{ j , \ell } \leq \dim_{ \K } \of{ V } $ the eigenvalue and rank of $ \vect{ e }_{ j , \ell } $, considered as a generalized eigenvector of $ \vect{ A } $, so that $ \of{ \vect{ A } - \mu_{ j , \ell } \id_{ V } }^{ r_{ j , \ell } } \of{ \vect{ e }_{ j , \ell } } = \vect{ 0 } $ for all $ 1 \leq j \leq m $ and $ 1 \leq \ell \leq \dim_{ \K } \of{ U_{ j } } $. In particular, we note that
\begin{align*}
\of{ \vect{ C }_{ \mathscr{ U } }^{ n } \vect{ A } \vect{ C }_{ \mathscr{ U } }^{ -n } - \mu_{ j , \ell } \id_{ V } }^{ r_{ j , \ell } } \of{ e^{ -\of{ j - \floor{ \frac{ m }{ 2 } } - 1 } n } \vect{ C }_{ \mathscr{ U } }^{ n } \of{ \vect{ e }_{ j , \ell } } } & = e^{ -\of{ j - \floor{ \frac{ m }{ 2 } } - 1 } n } \vect{ C }_{ \mathscr{ U } }^{ n } \of{ \of{ \vect{ A } - \mu_{ j , \ell } \id_{ V } }^{ r_{ j , \ell } } \of{ \vect{ e }_{ j , \ell } } } \\
& = e^{ -\of{ j - \floor{ \frac{ m }{ 2 } } - 1 } n } \vect{ C }_{ \mathscr{ U } }^{ n } \of{ \vect{ 0 } } = \vect{ 0 }
\end{align*}
for all $ n \in \N $. We observe in the limit that $ \of{ \vect{ B } - \mu_{ j , \ell } \id_{ V } }^{ r_{ j , \ell } } \of{ \pi_{ j } \of{ \vect{ e }_{ j , \ell } } } = \vect{ 0 } $; that is, $ \pi_{ j } \of{ \vect{ e }_{ j , \ell } } $ is a generalized eigenvector of $ \vect{ B } $ of eigenvalue $ \mu_{ j , \ell } $, although not necessarily of rank $ r_{ j , \ell } $. In fact, since
\[
\vect{ B }^{ n } \of{ \pi_{ j } \of{ \vect{ e }_{ j , \ell } } } = \of{ \bigoplus_{ i = 1 }^{ m }{ \vect{ B }_{ i }^{ n } } } \of{ \pi_{ j } \of{ \vect{ e }_{ j , \ell } } } = \sum_{ i = 1 }^{ m }{ \vect{ B }_{ i }^{ n } \of{ \pi_{ i } \of{ \pi_{ j } \of{ \vect{ e }_{ j , \ell } } } } } = \vect{ B }_{ j }^{ n } \of{ \pi_{ j } \of{ \vect{ e }_{ j , \ell } } }
\]
for all $ n \in \N $, $ \pi_{ j } \of{ \vect{ e }_{ j , \ell } } $ is also a generalized eigenvector of $ \vect{ B }_{ j } $ of eigenvalue $ \mu_{ j , \ell } $. In summary, we have shown that $ \pi_{ j } \of{ E_{ j } } $ is a basis for $ U_{ j } $ consisting of generalized eigenvectors of both $ \vect{ B } $ and $ \vect{ B }_{ j } $. We can further conclude that $ \bigsqcup_{ j = 1 }^{ m }{ \pi_{ j } \of{ E_{ j } } } $ is a basis for $ V = \bigoplus_{ j = 1 }^{ m }{ U_{ j } } $ consisting of generalized eigenvectors of $ \vect{ B } $. 

\begin{enumerate}

\item[(1)]

The desired results clearly hold if $ k = 0 $, so we may assume that $ 1 \leq k \leq \dim_{ \K } \of{ V } $. Note that $ E_{ m }' = \bigsqcup_{ i = 1 }^{ m }{ E_{ i } } $ is a basis for $ V $ which contains both $ E_{ j }' $ and $ E_{ m }' \cap \vect{ A }^{ \pm }_{ k } $, themselves bases for $ U_{ j }' $ and $ \vect{ A }^{ \pm }_{ k } $, respectively. Thus $ E_{ j }' \cap E_{ m }' \cap \vect{ A }^{ \pm }_{ k } = E_{ j }' \cap \vect{ A }^{ \pm }_{ k } $ is a basis for $ U_{ j }' \cap \vect{ A }^{ \pm }_{ k } $. We now note that
\begin{align*}
E_{ j }' \cap \vect{ A }^{ + }_{ k } & = \set{ \vect{ e }_{ i , \ell } \in U_{ j }' : 1 \leq i \leq j , 1 \leq \ell \leq \dim_{ \K } \of{ U_{ i } } , \textrm{ and } \abs{ \mu_{ i , \ell } } \geq \lambda_{ k } \of{ \vect{ A } } } \\
& = \set{ \vect{ e }_{ i , \ell } \in U_{ j }' : 1 \leq i \leq j , 1 \leq \ell \leq \dim_{ \K } \of{ U_{ i } } , \textrm{ and } \abs{ \mu_{ i , \ell } } \geq \lambda_{ \card{ E_{ j }' \cap \vect{ A }^{ + }_{ k } } } \of{ \vect{ A }_{ j }' } } \\
& = E_{ j }' \cap \of{ \vect{ A }_{ j }' }^{ + }_{ \card{ E_{ j }' \cap \vect{ A }^{ + }_{ k } } } = E_{ j }' \cap \of{ \vect{ A }_{ j }' }^{ + }_{ \dim_{ \K } \of{ U_{ j }' \cap \vect{ A }^{ + }_{ k } } }
\end{align*}
is also a basis for $ U_{ j }' \cap \of{ \vect{ A }_{ j }' }^{ + }_{ \dim_{ \K } \of{ U_{ j }' \cap \vect{ A }^{ + }_{ k } } } $, and similarly
\begin{align*}
E_{ j }' \cap \vect{ A }^{ - }_{ k } & = \set{ \vect{ e }_{ i , \ell } \in U_{ j }' : 1 \leq i \leq j , 1 \leq \ell \leq \dim_{ \K } \of{ U_{ i } } , \textrm{ and } \abs{ \mu_{ i , \ell } } \leq \lambda_{ \dim_{ \K } \of{ V } - k + 1 } \of{ \vect{ A } } } \\
& = \set{ \vect{ e }_{ i , \ell } \in U_{ j }' : 1 \leq i \leq j , 1 \leq \ell \leq \dim_{ \K } \of{ U_{ i } } , \textrm{ and } \abs{ \mu_{ i , \ell } } \leq \lambda_{ \dim_{ \K } \of{ U_{ j }' } - \card{ E_{ j }' \cap \vect{ A }^{ + }_{ k } } + 1 } \of{ \vect{ A }_{ j }' } } \\
& = E_{ j }' \cap \of{ \vect{ A }_{ j }' }^{ - }_{ \card{ E_{ j }' \cap \vect{ A }^{ - }_{ k } } } = E_{ j }' \cap \of{ \vect{ A }_{ j }' }^{ - }_{ \dim_{ \K } \of{ U_{ j }' \cap \vect{ A }^{ - }_{ k } } }
\end{align*}
is also a basis for $ U_{ j }' \cap \of{ \vect{ A }_{ j }' }^{ - }_{ \dim_{ \K } \of{ U_{ j }' \cap \vect{ A }^{ - }_{ k } } } $. We conclude that $ U_{ j }' \cap \vect{ A }^{ \pm }_{ k } = \of{ \vect{ A }_{ j }' }^{ \pm }_{ \dim_{ \K } \of{ U_{ j }' \cap \vect{ A }^{ \pm }_{ k } } } $.

Entirely analogously, $ \bigsqcup_{ i = 1 }^{ m }{ \pi_{ j } \of{ E_{ j } } } $ is a basis for $ V $ which contains both $ \pi_{ j } \of{ E_{ j } } $ and $ \of{ \bigsqcup_{ i = 1 }^{ m }{ \pi_{ j } \of{ E_{ j } } } } \cap \vect{ B }^{ \pm }_{ k } $, themselves bases for $ U_{ j } $ and $ \vect{ B }^{ \pm }_{ k } $, respectively. Thus $ \pi_{ j } \of{ E_{ j } } \cap \of{ \bigsqcup_{ i = 1 }^{ m }{ \pi_{ j } \of{ E_{ j } } } } \cap \vect{ B }^{ \pm }_{ k } = \pi_{ j } \of{ E_{ j } } \cap \vect{ B }^{ \pm }_{ k } $ is a basis for $ U_{ j } \cap \vect{ B }^{ \pm }_{ k } $. We now note that
\begin{align*}
\pi_{ j } \of{ E_{ j } } \cap \vect{ B }^{ + }_{ k } & = \set{ \pi_{ j } \of{ \vect{ e }_{ j , \ell } } \in U_{ j } : 1 \leq \ell \leq \dim_{ \K } \of{ U_{ j } } \textrm{ and } \abs{ \mu_{ j , \ell } } \geq \lambda_{ k } \of{ \vect{ B } } } \\
& = \set{ \pi_{ j } \of{ \vect{ e }_{ j , \ell } } \in U_{ j } : 1 \leq \ell \leq \dim_{ \K } \of{ U_{ j } } \textrm{ and } \abs{ \mu_{ j , \ell } } \geq \lambda_{ \card{ \pi_{ j } \of{ E_{ j } } \cap \vect{ B }^{ + }_{ k } } } \of{ \vect{ B }_{ j } } } \\
& = \pi_{ j } \of{ E_{ j } } \cap \of{ \vect{ B }_{ j } }^{ + }_{ \card{ \pi_{ j } \of{ E_{ j } } \cap \vect{ B }^{ + }_{ k } } } = \pi_{ j } \of{ E_{ j } } \cap \of{ \vect{ B }_{ j } }^{ + }_{ \dim_{ \K } \of{ U_{ j } \cap \vect{ B }^{ + }_{ k } } }
\end{align*}
is also a basis for $ U_{ j } \cap \vect{ B }^{ + }_{ k } $, and similarly
\begin{align*}
\pi_{ j } \of{ E_{ j } } \cap \vect{ B }^{ - }_{ k } & = \set{ \pi_{ j } \of{ \vect{ e }_{ j , \ell } } \in U_{ j } : 1 \leq \ell \leq \dim_{ \K } \of{ U_{ j } } \textrm{ and } \abs{ \mu_{ j , \ell } } \leq \lambda_{ \dim_{ \K } \of{ V } - k + 1 } \of{ \vect{ B } } } \\
& = \set{ \pi_{ j } \of{ \vect{ e }_{ j , \ell } } \in U_{ j } : 1 \leq \ell \leq \dim_{ \K } \of{ U_{ j } } \textrm{ and } \abs{ \mu_{ j , \ell } } \geq \lambda_{ \dim_{ \K } \of{ U_{ j } } - \card{ \pi_{ j } \of{ E_{ j } } \cap \vect{ B }^{ + }_{ k } } + 1 } \of{ \vect{ B }_{ j } } } \\
& = \pi_{ j } \of{ E_{ j } } \cap \of{ \vect{ B }_{ j } }^{ - }_{ \card{ \pi_{ j } \of{ E_{ j } } \cap \vect{ B }^{ - }_{ k } } } = \pi_{ j } \of{ E_{ j } } \cap \of{ \vect{ B }_{ j } }^{ - }_{ \dim_{ \K } \of{ U_{ j } \cap \vect{ B }^{ - }_{ k } } }
\end{align*}
is also a basis for $ U_{ j } \cap \vect{ B }^{ - }_{ k } $. We conclude that $ U_{ j } \cap \vect{ B }^{ \pm }_{ k } = \of{ \vect{ B }_{ j } }^{ \pm }_{ \dim_{ \K } \of{ U_{ j } \cap \vect{ B }^{ \pm }_{ k } } } $.

It remains to show $ \pi_{ j } \of{ U_{ j }' \cap \vect{ A }^{ \pm }_{ k } } = U_{ j } \cap \vect{ B }^{ \pm }_{ k } $. To this end, we see that
\begin{align*}
\pi_{ j } \of{ U_{ j }' \cap \vect{ A }^{ + }_{ k } } & = \pi_{ j } \of{ \lspan_{ \K } \of{ E_{ j }' \cap \vect{ A }^{ + }_{ k } } } = \lspan_{ \K } \of{ \pi_{ j } \of{ E_{ j }' \cap \vect{ A }^{ + }_{ k } } } \\
& = \lspan_{ \K } \of{ \set{ \pi_{ j } \of{ \vect{ e }_{ i , \ell } } \in U_{ j } : 1 \leq i \leq j , 1 \leq \ell \leq \dim_{ \K } \of{ U_{ i } } , \textrm{ and } \abs{ \mu_{ i , \ell } } \geq \lambda_{ k } \of{ \vect{ A } } } } \\
& = \lspan_{ \K } \of{ \set{ \pi_{ j } \of{ \vect{ e }_{ j , \ell } } \in U_{ j } : 1 \leq \ell \leq \dim_{ \K } \of{ U_{ j } } , \textrm{ and } \abs{ \mu_{ j , \ell } } \geq \lambda_{ k } \of{ \vect{ B } } } } \\
& = \lspan_{ \K } \of{ \pi_{ j } \of{ E_{ j } } \cap \vect{ B }^{ + }_{ k } } = U_{ j } \cap \vect{ B }^{ + }_{ k } ,
\end{align*}
and similarly
\begin{align*}
\pi_{ j } \of{ U_{ j }' \cap \vect{ A }^{ - }_{ k } } & = \pi_{ j } \of{ \lspan_{ \K } \of{ E_{ j }' \cap \vect{ A }^{ - }_{ k } } } = \lspan_{ \K } \of{ \pi_{ j } \of{ E_{ j }' \cap \vect{ A }^{ - }_{ k } } } \\
& = \lspan_{ \K } \of{ \set{ \pi_{ j } \of{ \vect{ e }_{ i , \ell } } \in U_{ j } : 1 \leq i \leq j , 1 \leq \ell \leq \dim_{ \K } \of{ U_{ i } } , \textrm{ and } \abs{ \mu_{ i , \ell } } \leq \lambda_{ \dim_{ \K } \of{ V } - k + 1 } \of{ \vect{ A } } } } \\
& = \lspan_{ \K } \of{ \set{ \pi_{ j } \of{ \vect{ e }_{ j , \ell } } \in U_{ j } : 1 \leq \ell \leq \dim_{ \K } \of{ U_{ j } } , \textrm{ and } \abs{ \mu_{ j , \ell } } \leq \lambda_{ \dim_{ \K } \of{ V } - k + 1 } \of{ \vect{ B } } } } \\
& = \lspan_{ \K } \of{ \pi_{ j } \of{ E_{ j } } \cap \vect{ B }^{ - }_{ k } } = U_{ j } \cap \vect{ B }^{ - }_{ k } .
\end{align*}

\item[(2)]

We previously argued that for each $ 1 \leq j \leq m $, the projection $ \pi_{ j } \colon V \onto U_{ j } $ onto the $ j $th factor $ U_{ j } $ restricts to a linear isomorphism $ \lspan_{ \K } \of{ E_{ j } } \to U_{ j } $. The further restriction to $ \lspan_{ \K } \of{ E_{ j } \cap \vect{ A }^{ \pm }_{ k } } $ is injective and therefore an isomorphism onto its image
\begin{align*}
\pi_{ j } \of{ \lspan_{ \K } \of{ E_{ j } \cap \vect{ A }^{ \pm }_{ k } } } & = \lspan_{ \K } \of{ \pi_{ j } \of{ E_{ j } \cap \vect{ A }^{ \pm }_{ k } } } = \lspan_{ \K } \of{ \pi_{ j } \of{ E_{ j }' \cap \vect{ A }^{ \pm }_{ k } } } \\
& = \pi_{ j } \of{ \lspan_{ \K } \of{ E_{ j }' \cap \vect{ A }^{ \pm }_{ k } } } = \pi_{ j } \of{ U_{ j }' \cap \vect{ A }^{ \pm }_{ k } } .
\end{align*}
The computations above in (1) now imply that
\begin{align*}
\vect{ A }^{ \pm }_{ k } & = \lspan_{ \K } \of{ \of{ \bigsqcup_{ j = 1 }^{ m }{ E_{ j } } } \cap \vect{ A }^{ \pm }_{ k } } = \lspan_{ \K } \of{ \bigsqcup_{ j = 1 }^{ m }{ E_{ j } \cap \vect{ A }^{ \pm }_{ k } } } = \bigoplus_{ j = 1 }^{ m }{ \lspan_{ \K } \of{ E_{ j } \cap \vect{ A }^{ \pm }_{ k } } } \\
& \cong \bigoplus_{ j = 1 }^{ m }{ \pi_{ j } \of{ \lspan_{ \K } \of{ E_{ j } \cap \vect{ A }^{ \pm }_{ k } } } } = \bigoplus_{ j = 1 }^{ m }{ \pi_{ j } \of{ U_{ j }' \cap \vect{ A }^{ \pm }_{ k } } } = \bigoplus_{ j = 1 }^{ m }{ U_{ j } \cap \vect{ B }^{ \pm }_{ k } } = \lspan_{ \K } \of{ \bigsqcup_{ j = 1 }^{ m }{ \pi_{ j } \of{ E_{ j } } \cap \vect{ B }^{ \pm }_{ k } } } \\
& = \lspan_{ \K } \of{ \of{ \bigsqcup_{ j = 1 }^{ m }{ \pi_{ j } \of{ E_{ j } } } } \cap \vect{ B }^{ \pm }_{ k } } = \vect{ B }^{ \pm }_{ k } .
\end{align*}

\item[(3)]

The desired results clearly holds if $ k = 0 $ or $ k = \dim_{ \K } \of{ V } $, so we may assume that $ k \in \Delta_{ V } $. For each $ 1 \leq i \leq j $ and $ 1 \leq \ell \leq \dim_{ \K } \of{ U_{ i } } $, $ \mu_{ i , \ell } $ is an eigenvalue of $ \vect{ A } $, and so either $ \abs{ \mu_{ i , \ell } } \geq \lambda_{ k } \of{ \vect{ A } } $ or $ \abs{ \mu_{ i , \ell } } \leq \lambda_{ k + 1 } \of{ \vect{ A } } $ by the definition of eigenvalue magnitudes. Thus
\begin{align*}
E_{ j }' & = \set{ \vect{ e }_{ i , \ell } \in U_{ j }' : 1 \leq i \leq j \textrm{ and } 1 \leq \ell \leq \dim_{ \K } \of{ U_{ i } } } \\
& = \set{ \vect{ e }_{ i , \ell } \in U_{ j }' : 1 \leq i \leq j , 1 \leq \ell \leq \dim_{ \K } \of{ U_{ i } } , \textrm{ and } \abs{ \mu_{ j , \ell } } \geq \lambda_{ k } \of{ \vect{ A } } } \\
& \qquad \cup \set{ \vect{ e }_{ i , \ell } \in U_{ j }' : 1 \leq i \leq j , 1 \leq \ell \leq \dim_{ \K } \of{ U_{ i } } , \textrm{ and } \abs{ \mu_{ j , \ell } } \leq \lambda_{ k + 1 } \of{ \vect{ A } } } \\
& = \of{ E_{ j }' \cap \vect{ A }^{ + }_{ k } } \cup \of{ E_{ j }' \cap \vect{ A }^{ - }_{ \iota_{ V } \of{ k } } } .
\end{align*}
We conclude that
\begin{align*}
U_{ j }' & = \lspan_{ \K } \of{ E_{ j }' } = \lspan_{ \K } \of{ \of{ E_{ j }' \cap \vect{ A }^{ + }_{ k } } \cup \of{ E_{ j }' \cap \vect{ A }^{ - }_{ \iota_{ V } \of{ k } } } } \\
& = \lspan_{ \K } \of{ \lspan_{ \K } \of{ E_{ j }' \cap \vect{ A }^{ + }_{ k } } , \lspan_{ \K } \of{ E_{ j }' \cap \vect{ A }^{ - }_{ \iota_{ V } \of{ k } } } } = \lspan_{ \K } \of{ U_{ j }' \cap \vect{ A }^{ + }_{ k } , U_{ j }' \cap \vect{ A }^{ - }_{ \iota_{ V } \of{ k } } } .
\end{align*}

Entirely analogously, for each $ 1 \leq \ell \leq \dim_{ \K } \of{ U_{ j } } $, $ \mu_{ j , \ell } $ is an eigenvalue of $ \vect{ B } $, and so either $ \abs{ \mu_{ j , \ell } } \geq \lambda_{ k } \of{ \vect{ B } } $ or $ \abs{ \mu_{ j , \ell } } \leq \lambda_{ k + 1 } \of{ \vect{ B } } $. Thus
\begin{align*}
\pi_{ j } \of{ E_{ j } } & = \set{ \pi_{ j } \of{ \vect{ e }_{ j , \ell } } \in U_{ j } : 1 \leq \ell \leq \dim_{ \K } \of{ U_{ j } } } \\
& = \set{ \pi_{ j } \of{ \vect{ e }_{ j , \ell } } \in U_{ j } : 1 \leq \ell \leq \dim_{ \K } \of{ U_{ j } } \textrm{ and } \abs{ \mu_{ j , \ell } } \geq \lambda_{ k } \of{ \vect{ B } } } \\
& \qquad \cup \set{ \pi_{ j } \of{ \vect{ e }_{ j , \ell } } \in U_{ j } : 1 \leq \ell \leq \dim_{ \K } \of{ U_{ j } } \textrm{ and } \abs{ \mu_{ j , \ell } } \leq \lambda_{ k + 1 } \of{ \vect{ B } } } \\
& = \of{ \pi_{ j } \of{ E_{ j } } \cap \vect{ B }^{ + }_{ k } } \cup \of{ \pi_{ j } \of{ E_{ j } } \cap \vect{ B }^{ - }_{ \iota_{ V } \of{ k } } } .
\end{align*}
We conclude that
\begin{align*}
U_{ j } & = \lspan_{ \K } \of{ \pi_{ j } \of{ E_{ j } } } = \lspan_{ \K } \of{ \of{ \pi_{ j } \of{ E_{ j } } \cap \vect{ B }^{ + }_{ k } } \cup \of{ \pi_{ j } \of{ E_{ j } } \cap \vect{ B }^{ - }_{ \iota_{ V } \of{ k } } } } \\
& = \lspan_{ \K } \of{ \lspan_{ \K } \of{ \pi_{ j } \of{ E_{ j } } \cap \vect{ B }^{ + }_{ k } } , \lspan_{ \K } \of{ \pi_{ j } \of{ E_{ j } } \cap \vect{ B }^{ - }_{ \iota_{ V } \of{ k } } } } = \lspan_{ \K } \of{ U_{ j } \cap \vect{ B }^{ + }_{ k } , U_{ j } \cap \vect{ B }^{ - }_{ \iota_{ V } \of{ k } } } .
\end{align*}

\item[(4)]

We argued above that $ \lambda_{ k } \of{ \vect{ A } } = \lambda_{ k } \of{ \vect{ B } } $. The above computations in (1) imply that
\begin{align*}
\lambda_{ k } \of{ \vect{ B } } & = \min \of{ \set{ \abs{ \mu_{ i , \ell } } : 1 \leq i \leq m , 1 \leq \ell \leq \dim_{ \K } \of{ U_{ i } } , \textrm{ and } \abs{ \mu_{ i , \ell } } \geq \lambda_{ k } \of{ \vect{ B } } } } \\
& = \min_{ i }{ \min \of{ \set{ \abs{ \mu_{ i , \ell } } : 1 \leq \ell \leq \dim_{ \K } \of{ U_{ i } } \textrm{ and } \abs{ \mu_{ i , \ell } } \geq \lambda_{ k } \of{ \vect{ B } } } } } \\
& = \min_{ i }{ \min \of{ \set{ \abs{ \mu_{ i , \ell } } : 1 \leq \ell \leq \dim_{ \K } \of{ U_{ i } } \textrm{ and } \abs{ \mu_{ i , \ell } } \geq \lambda_{ \dim_{ \K } \of{ U_{ i } \cap \vect{ B }^{ + }_{ k } } } \of{ \vect{ B }_{ i } } } } } \\
& = \min_{ i }{ \lambda_{ \dim_{ \K } \of{ U_{ i } \cap \vect{ B }^{ + }_{ k } } } \of{ \vect{ B }_{ i } } } = \min_{ i }{ \lambda_{ \dim_{ \K } \of{ \pi_{ i } \of{ U_{ i }' \cap \vect{ A }^{ + }_{ k } } } } \of{ \vect{ B }_{ i } } } ,
\end{align*}
where the above minimum is indexed by all $ 1 \leq i \leq m $ so that
\[
\dim_{ \K } \of{ \pi_{ i } \of{ U_{ i }' \cap \vect{ A }^{ + }_{ k } } } = \dim_{ \K } \of{ U_{ i } \cap \vect{ B }^{ + }_{ k } } = \card{ \pi_{ i } \of{ E_{ i } } \cap \vect{ B }^{ + }_{ k } } > 0 .
\]
Similarly
\begin{align*}
\lambda_{ k } \of{ \vect{ B } } & = \max \of{ \set{ \abs{ \mu_{ j , \ell } } : 1 \leq j \leq m , 1 \leq \ell \leq \dim_{ \K } \of{ U_{ j } } , \textrm{ and } \abs{ \mu_{ j , \ell } } \leq \lambda_{ k } \of{ \vect{ B } } } } \\
& = \max_{ j }{ \max \of{ \set{ \abs{ \mu_{ j , \ell } } : 1 \leq \ell \leq \dim_{ \K } \of{ U_{ j } } \textrm{ and } \abs{ \mu_{ j , \ell } } \leq \lambda_{ k } \of{ \vect{ B } } } } } \\
& = \max_{ j }{ \max \of{ \set{ \abs{ \mu_{ j , \ell } } : 1 \leq \ell \leq \dim_{ \K } \of{ U_{ j } } \textrm{ and } \abs{ \mu_{ j , \ell } } \leq \lambda_{ \dim_{ \K } \of{ U_{ j } } - \dim_{ \K } \of{ U_{ j } \cap \vect{ B }^{ - }_{ \iota_{ V } \of{ k } + 1 } } + 1 } \of{ \vect{ B }_{ j } } } } } \\
& = \max_{ j }{ \lambda_{ \dim_{ \K } \of{ U_{ j } } - \dim_{ \K } \of{ U_{ j } \cap \vect{ B }^{ - }_{ \iota_{ V } \of{ k } + 1 } } + 1 } \of{ \vect{ B }_{ j } } } = \max_{ j }{ \lambda_{ \dim_{ \K } \of{ U_{ j } } - \dim_{ \K } \of{ U_{ j }' \cap \vect{ A }^{ - }_{ \iota_{ V } \of{ k } + 1 } } + 1 } \of{ \vect{ B }_{ j } } } ,
\end{align*}
where the above maximum is indexed by all $ 1 \leq j \leq m $ so that
\[
\dim_{ \K } \of{ \pi_{ j } \of{ U_{ j }' \cap \vect{ A }^{ - }_{ \iota_{ V } \of{ k } + 1 } } } = \dim_{ \K } \of{ U_{ j } \cap \vect{ B }^{ - }_{ \iota_{ V } \of{ k } + 1 } } = \card{ \pi_{ j } \of{ E_{ j } } \cap \vect{ B }^{ - }_{ \iota_{ V } \of{ k } + 1 } } > 0 .
\]

\item[(5)]

Recall that $ \of{ \vect{ B }_{ j } }^{ \pm }_{ \dim_{ \K } \of{ U_{ j } \cap \vect{ B }^{ \pm }_{ k } } } = U_{ j } \cap \vect{ B }^{ \pm }_{ k } $, as argued above in (1). If
\[
\dim_{ \K } \of{ U_{ j } \cap \vect{ B }^{ + }_{ k } } = \dim_{ \K } \of{ \pi_{ j } \of{ U_{ j }' \cap \vect{ A }^{ + }_{ k } } } \in \Delta_{ U_{ j } } ,
\]
then
\[
\dim_{ \K } \of{ \of{ \vect{ B }_{ j } }^{ + }_{ \dim_{ \K } \of{ U_{ j }' \cap \vect{ A }^{ + }_{ k } } } } = \dim_{ \K } \of{ \of{ \vect{ B }_{ j } }^{ + }_{ \dim_{ \K } \of{ U_{ j } \cap \vect{ B }^{ + }_{ k } } } } = \dim_{ \K } \of{ U_{ j } \cap \vect{ B }^{ + }_{ k } } = \dim_{ \K } \of{ U_{ j }' \cap \vect{ A }^{ + }_{ k } } ,
\]
and so $ \vect{ B }_{ j } $ is $ P_{ \dim_{ \K } \of{ U_{ j }' \cap \vect{ A }^{ + }_{ k } } } $-proximal by \autoref{fact:proximalityCharacterizations}.

Similarly, if
\[
\dim_{ \K } \of{ U_{ j } \cap \vect{ B }^{ - }_{ k } } = \dim_{ \K } \of{ \pi_{ j } \of{ U_{ j }' \cap \vect{ A }^{ - }_{ k } } } \in \Delta_{ U_{ j } } ,
\]
then
\begin{align*}
\dim_{ \K } \of{ \of{ \vect{ B }_{ j } }^{ - }_{ \iota_{ U_{ j } } \of{ \iota_{ U_{ j } } \of{ \dim_{ \K } \of{ U_{ j }' \cap \vect{ A }^{ - }_{ k } } } } } } & = \dim_{ \K } \of{ \of{ \vect{ B }_{ j } }^{ - }_{ \iota_{ U_{ j } } \of{ \iota_{ U_{ j } } \of{ \dim_{ \K } \of{ U_{ j } \cap \vect{ B }^{ - }_{ k } } } } } } \\
& = \dim_{ \K } \of{ \of{ \vect{ B }_{ j } }^{ - }_{ \dim_{ \K } \of{ U_{ j } \cap \vect{ B }^{ - }_{ k } } } } = \dim_{ \K } \of{ U_{ j } \cap \vect{ B }^{ - }_{ k } } \\
& = \dim_{ \K } \of{ U_{ j }' \cap \vect{ A }^{ - }_{ k } } = \iota_{ U_{ j } } \of{ \iota_{ U_{ j } } \of{ \dim_{ \K } \of{ U_{ j }' \cap \vect{ A }^{ - }_{ k } } } } ,
\end{align*}
and so $ \vect{ B }_{ j } $ is $ P_{ \iota_{ U_{ j } } \of{ \dim_{ \K } \of{ U_{ j }' \cap \vect{ A }^{ - }_{ k } } } } $-proximal by \autoref{fact:proximalityCharacterizations}. We conclude that $ \vect{ B }_{ j } $ is $ P_{ \theta_{ j } } $-biproximal by \autoref{fact:biproximality}. \qedhere

\end{enumerate}

\end{proof}

\end{proposition+}

The statements and proof of \autoref{prop:reducibleEigenvalues} above hint at the importance of the numbers
\[
\dim_{ \K } \of{ \pi_{ j } \of{ U_{ j }' \cap \vect{ A }^{ + }_{ k } } } = \dim_{ \K } \of{ U_{ j } \cap \vect{ B }^{ + }_{ k } }
\]
for a linear transformation $ \vect{ A } \colon V \to V $ which is block upper triangular relative to $ \mathscr{ U } $ and its block diagonalization $ \vect{ B } $. If such a linear transformation $ \vect{ A } $ is in the image of a reducible Anosov representation preserving the flag associated to $ \mathscr{ U } $, then these numbers play a vital role in computing the logarithmic eigenvalue gaps and Anosov limit maps. We formalize this role in the following definitions and in \autoref{thm:eigenvalueConfiguration}.

\begin{definition*}[Admissible family; structured flag]

A family $ Q = \of{ q_{ j , k } }_{ j , k } $ of integers $ 0 \leq q_{ j , k } \leq \dim_{ \K } \of{ U_{ j } } $ indexed by $ 1 \leq j \leq m $ and $ k \in \theta $ is called $ \of{ \mathscr{ U } , \theta } $-\emph{admissible} for a subset $ \theta \subseteq \Delta_{ V } $ if
\[
\sum_{ j = 1 }^{ m }{ q_{ j , k } } = k
\]
for all $ k \in \theta $.

Given such a $ \of{ \mathscr{ U } , \theta } $-admissible family $ Q = \of{ q_{ j , k } }_{ j , k } $, a flag $ F = \of{ F_{ k } }_{ k \in \theta } $ in $ V $ of signature $ \theta $ is called \emph{weakly $ \of{ \mathscr{ U } , \theta , Q } $-structured} if $ \dim_{ \K } \of{ \pi_{ j } \of{ U_{ j }' \cap F_{ k } } } = q_{ j , k } $ for all $ 1 \leq j \leq m $ and $ k \in \theta $. $ F $ is called \emph{strongly $ \of{ \mathscr{ U } , \theta , Q } $-structured} if $ \dim_{ \K } \of{ U_{ j } \cap F_{ k } } = q_{ j , k } $ for all $ 1 \leq j \leq m $ and $ k \in \theta $. If $ F $ is strongly $ \of{ \mathscr{ U } , \theta , Q } $-structured, then
\[
F_{ k } = \bigoplus_{ j = 1 }^{ m }{ U_{ j } \cap F_{ k } }
\]
for all $ k \in \theta $, so that $ F $ is also weakly $ \of{ \mathscr{ U } , \theta , Q } $-structured. We denote by $ W^{ \mathscr{ U } , \theta }_{ Q } \of{ V } $ and $ S^{ \mathscr{ U } , \theta }_{ Q } \of{ V } $ the closed subsets of the flag variety $ \Flag_{ \theta } \of{ V } $ consisting of weakly and strongly $ \of{ \mathscr{ U } , \theta , Q } $-structured flags, respectively.

\end{definition*}

\begin{lemma+} \label{lem:nonEmptyAdmissibleCondition}

Fix a subset $ \theta \subseteq \Delta_{ V } $ and a $ \of{ \mathscr{ U } , \theta } $-admissible family $ Q = \of{ q_{ j , k } }_{ j , k } $. For each $ k \in \theta $, there are $ 1 \leq i , j \leq m $ so that $ q_{ i , k } > 0 $ and $ q_{ j , k } < \dim_{ \K } \of{ U_{ j } } $.

\begin{proof}

Recall that $ \Delta_{ V } = \set{ 1 , \dotsc , \dim_{ \K } \of{ V } - 1 } $, so that $ 0 < k < \dim_{ \K } \of{ V } $. If $ q_{ i , k } = 0 $ for all $ 1 \leq i \leq m $, then
\[
k = \sum_{ i = 1 }^{ m }{ q_{ i , k } } = \sum_{ i = 1 }^{ m }{ 0 } = 0 .
\]
On the other hand, if $ q_{ j , k } = \dim_{ \K } \of{ U_{ j } } $ for all $ 1 \leq j \leq m $, then
\[
k = \sum_{ j = 1 }^{ m }{ q_{ j , k } } = \sum_{ j = 1 }^{ m }{ \dim_{ \K } \of{ U_{ j } } } = \dim_{ \K } \of{ \bigoplus_{ j = 1 }^{ m }{ U_{ j } } } = \dim_{ \K } \of{ V } .
\]
Thus there are $ 1 \leq i , j \leq m $ so that $ q_{ i , k } > 0 $ and $ q_{ j , k } < \dim_{ \K } \of{ U_{ j } } $. \qedhere

\end{proof}

\end{lemma+}

Equipped with this new language of admissible families and structured flags, we derive the following corollary of \autoref{prop:reducibleEigenvalues} about the proximality of reducible linear transformations:

\begin{corollary+} \label{cor:reducibleProximality}

Fix a linear transformation $ \vect{ A } \colon V \to V $ which is block upper triangular relative to $ \mathscr{ U } $, and consider its block diagonalization $ \vect{ B } = \bigoplus_{ j = 1 }^{ m }{ \vect{ B }_{ j } } $. For any subset $ \theta \subseteq \Delta_{ V } $, $ \vect{ A } $ is $ P_{ \theta } $-proximal if and only if
\[
U_{ j } = \pi_{ j } \of{ U_{ j }' \cap \vect{ A }^{ + }_{ k } } \oplus \pi_{ j } \of{ U_{ j }' \cap \vect{ A }^{ - }_{ \iota_{ V } \of{ k } } }
\]
for all $ 1 \leq j \leq m $ and $ k \in \theta $, in which case the following also hold:
\begin{enumerate}

\item[(1)]

The family $ Q = \of{ q_{ j , k } }_{ j , k } $ indexed by $ 1 \leq j \leq m $ and $ k \in \theta $ and defined by the formula
\[
q_{ j , k } = \dim_{ \K } \of{ \pi_{ j } \of{ U_{ j }' \cap \vect{ A }^{ + }_{ k } } } = \dim_{ \K } \of{ U_{ j } \cap \vect{ B }^{ + }_{ k } }
\]
is a $ \of{ \mathscr{ U } , \theta } $-admissible family called the \emph{large eigenvalue $ \theta $-configuration} of $ \vect{ A } $ relative to $ \mathscr{ U } $.

\item[(2)]

The attracting $ \theta $-flags $ \vect{ A }^{ + }_{ \theta } $ and $ \vect{ B }^{ + }_{ \theta } $ are weakly and strongly $ \of{ \mathscr{ U } , \theta , Q } $-structured, respectively.

\item[(3)]

$ \vect{ A } $ and $ \vect{ B } $ have logarithmic eigenvalue gaps
\begin{align*}
\log \of{ \frac{ \lambda_{ k } \of{ \vect{ A } } }{ \lambda_{ k + 1 } \of{ \vect{ A } } } } = \log \of{ \frac{ \lambda_{ k } \of{ \vect{ B } } }{ \lambda_{ k + 1 } \of{ \vect{ B } } } } = \min_{ i , j }{ \log \of{ \frac{ \lambda_{ q_{ i , k } } \of{ \vect{ B }_{ i } } }{ \lambda_{ q_{ j , k } + 1 } \of{ \vect{ B }_{ j } } } } } 
\end{align*}
for each $ k \in \Delta_{ V } $, where the above minimum is indexed by all $ 1 \leq i , j \leq m $ so that $ q_{ i , k } > 0 $ and $ q_{ j , k } < \dim_{ \K } \of{ U_{ j } } $.

\end{enumerate}

\begin{proof}

Suppose first that $ U_{ j } = \pi_{ j } \of{ U_{ j }' \cap \vect{ A }^{ + }_{ k } } \oplus \pi_{ j } \of{ U_{ j }' \cap \vect{ A }^{ - }_{ \iota_{ V } \of{ k } } } $ for all $ 1 \leq j \leq m $ and $ k \in \theta $. One can check from the computations in the above proof of \autoref{prop:reducibleEigenvalues} that
\begin{align*}
\vect{ B }^{ + }_{ k } \cap \vect{ B }^{ - }_{ \iota_{ V } \of{ k } } & = \bigoplus_{ j = 1 }^{ m }{ U_{ j } \cap \of{ \vect{ B }^{ + }_{ k } \cap \vect{ B }^{ - }_{ \iota_{ V } \of{ k } } } } = \bigoplus_{ j = 1 }^{ m }{ \of{ U_{ j } \cap \vect{ B }^{ + }_{ k } } \cap \of{ U_{ j } \cap \vect{ B }^{ - }_{ \iota_{ V } \of{ k } } } } \\
& = \bigoplus_{ j = 1 }^{ m }{ \pi_{ j } \of{ U_{ j }' \cap \vect{ A }^{ + }_{ k } } \cap \pi_{ j } \of{ U_{ j }' \cap \vect{ A }^{ - }_{ \iota_{ V } \of{ k } } } } = \bigoplus_{ j = 1 }^{ m }{ \set{ \vect{ 0 } } } = \set{ \vect{ 0 } } ,
\end{align*}
so that \autoref{fact:proximalityCharacterizations} implies that $ \vect{ B } $ is $ P_{ \theta } $-proximal. Since $ \vect{ A } $ and $ \vect{ B } $ have the same eigenvalue magnitudes, this implies that $ \vect{ A } $ is $ P_{ \theta } $-proximal.

Conversely, now suppose that $ \vect{ A } $ (and therefore $ \vect{ B } $) is $ P_{ \theta } $-proximal, and fix $ k \in \theta $. (1) and (2) in \autoref{prop:reducibleEigenvalues} imply that
\[
U_{ j } = \lspan_{ \K } \of{ \pi_{ j } \of{ U_{ j }' \cap \vect{ A }^{ + }_{ k } } , \pi_{ j } \of{ U_{ j }' \cap \vect{ A }^{ - }_{ \iota_{ V } \of{ k } } } } = \lspan_{ \K } \of{ U_{ j } \cap \vect{ B }^{ + }_{ k } , U_{ j } \cap \vect{ B }^{ - }_{ \iota_{ V } \of{ k } } }
\]
for all $ 1 \leq j \leq m $. On the other hand, \autoref{fact:proximalityCharacterizations} and (1) in \autoref{prop:reducibleEigenvalues} imply that
\[
\pi_{ j } \of{ U_{ j }' \cap \vect{ A }^{ + }_{ k } } \cap \pi_{ j } \of{ U_{ j }' \cap \vect{ A }^{ - }_{ \iota_{ V } \of{ k } } } = \of{ U_{ j } \cap \vect{ B }^{ + }_{ k } } \cap \of{ U_{ j } \cap \vect{ B }^{ - }_{ \iota_{ V } \of{ k } } } = U_{ j } \cap \vect{ B }^{ + }_{ k } \cap \vect{ B }^{ - }_{ \iota_{ V } \of{ k } } = U_{ j } \cap \set{ \vect{ 0 } } = \set{ \vect{ 0 } } .
\]
Thus $ U_{ j } = \pi_{ j } \of{ U_{ j }' \cap \vect{ A }^{ + }_{ k } } \oplus \pi_{ j } \of{ U_{ j }' \cap \vect{ A }^{ - }_{ \iota_{ V } \of{ k } } } $.

\begin{enumerate}

\item[(1)]

The computations presented above in the proof of \autoref{prop:reducibleEigenvalues} imply that
\begin{align*}
\sum_{ j = 1 }^{ m }{ q_{ j , k } } & = \sum_{ j = 1 }^{ m }{ \dim_{ \K } \of{ \pi_{ j } \of{ U_{ j }' \cap \vect{ A }^{ + }_{ k } } } } = \sum_{ j = 1 }^{ m }{ \dim_{ \K } \of{ U_{ j } \cap \vect{ B }^{ + }_{ k } } } = \dim_{ \K } \of{ \bigoplus_{ j = 1 }^{ m }{ U_{ j } \cap \vect{ B }^{ + }_{ k } } } = \dim_{ \K } \of{ \vect{ B }^{ + }_{ k } } = k
\end{align*}
for all $ k \in \theta $, and so $ Q = \of{ q_{ j , k } }_{ j , k } $ is $ \of{ \mathscr{ U } , \theta } $-admissible.

\item[(2)]

Note that $ q_{ j , k } = \dim_{ \K } \of{ \pi_{ j } \of{ U_{ j }' \cap \vect{ A }^{ + }_{ k } } } = \dim_{ \K } \of{ U_{ j } \cap \vect{ B }^{ + }_{ k } } $ for all $ 1 \leq j \leq m $ and $ k \in \theta $ by definition, so that $ \vect{ A }^{ + }_{ \theta } = \of{ \vect{ A }^{ + }_{ k } }_{ k \in \theta } $ and $ \vect{ B }^{ + }_{ \theta } = \of{ \vect{ B }^{ + }_{ k } }_{ k \in \theta } $ are weakly and strongly $ \of{ \mathscr{ U } , \theta , Q } $-structured, respectively.

\item[(3)]

Recall that
\begin{align*}
\dim_{ \K } \of{ U_{ j } } & = \dim_{ \K } \of{ \pi_{ j } \of{ U_{ j }' \cap \vect{ A }^{ + }_{ k } } \oplus \pi_{ j } \of{ U_{ j }' \cap \vect{ A }^{ - }_{ \iota_{ V } \of{ k } } } } = \dim_{ \K } \of{ \pi_{ j } \of{ U_{ j }' \cap \vect{ A }^{ + }_{ k } } } + \dim_{ \K } \of{ \pi_{ j } \of{ U_{ j }' \cap \vect{ A }^{ - }_{ \iota_{ V } \of{ k } } } } \\
& = q_{ j , k } + \dim_{ \K } \of{ \pi_{ j } \of{ U_{ j }' \cap \vect{ A }^{ - }_{ \iota_{ V } \of{ k } } } } .
\end{align*}
To derive the desired formula for the $ k $th logarithmic eigenvalue gap of $ \vect{ A } $, we further observe that (4) in \autoref{prop:reducibleEigenvalues} implies that $ \vect{ A } $ and $ \vect{ B } $ have eigenvalue magnitudes
\[
\lambda_{ k } \of{ \vect{ A } } = \lambda_{ k } \of{ \vect{ B } } = \min_{ \dim_{ \K } \of{ \pi_{ i } \of{ U_{ i }' \cap \vect{ A }^{ + }_{ k } } } > 0 }{ \lambda_{ \dim_{ \K } \of{ \pi_{ i } \of{ U_{ i }' \cap \vect{ A }^{ + }_{ k } } } } \of{ \vect{ B }_{ i } } } = \min_{ q_{ i , k } > 0 }{ \lambda_{ q_{ i , k } } \of{ \vect{ B }_{ i } } }
\]
and
\begin{align*}
\lambda_{ k + 1 } \of{ \vect{ A } } & = \lambda_{ k + 1 } \of{ \vect{ B } } = \max_{ \dim_{ \K } \of{ \pi_{ j } \of{ U_{ j }' \cap \vect{ A }^{ - }_{ \iota_{ V } \of{ k } } } } < \dim_{ \K } \of{ U_{ j } } }{ \lambda_{ \dim_{ \K } \of{ U_{ j } } - \dim_{ \K } \of{ U_{ j }' \cap \vect{ A }^{ - }_{ \iota_{ V } \of{ k } } } + 1 } \of{ \vect{ B }_{ j } } } \\
& = \max_{ q_{ j , k } < \dim_{ \K } \of{ U_{ j }' \cap \vect{ A }^{ + }_{ k } } }{ \lambda_{ q_{ j , k } + 1 } \of{ \vect{ B }_{ j } } } .
\end{align*}
Thus $ \vect{ A } $ and $ \vect{ B } $ have logarithmic eigenvalue gaps
\[
\log \of{ \frac{ \lambda_{ k } \of{ \vect{ A } } }{ \lambda_{ k + 1 } \of{ \vect{ A } } } } = \log \of{ \frac{ \lambda_{ k } \of{ \vect{ B } } }{ \lambda_{ k + 1 } \of{ \vect{ B } } } } = \min_{ i , j }{ \log \of{ \frac{ \lambda_{ q_{ i , k } } \of{ \vect{ B }_{ i } } }{ \lambda_{ q_{ j , k } + 1 } \of{ \vect{ B }_{ j } } } } }
\]
for all $ k \in \Delta_{ V } $, where the above minimum is indexed by all $ 1 \leq i , j \leq m $ so that $ q_{ i , k } > 0 $ and $ q_{ j , k } < \dim_{ \K } \of{ U_{ j } } $. \qedhere

\end{enumerate}

\end{proof}

\end{corollary+}

\subsection{Eigenvalues of Reducible Representations}

Armed with the tools and terminology developed above in \autoref{prop:reducibleEigenvalues} and \autoref{cor:reducibleProximality}, we are now ready to prove \autoref{thm:newEigenvalueConfiguration}. We will prove a more general result, applicable to all block upper triangular representations; \autoref{thm:newEigenvalueConfiguration} is recovered from \autoref{thm:eigenvalueConfiguration} as an immediate corollary for block diagonal representations.

\begin{theorem+} \label{thm:eigenvalueConfiguration}

Fix a representation $ \rho \colon \Gamma \to \GL \of{ V } $ and a subset $ \theta \subseteq \Delta_{ V } $. If $ \rho $ is block upper triangular relative to $ \mathscr{ U } $, with block diagonalization $ \eta = \bigoplus_{ j = 1 }^{ m }{ \eta_{ j } } $, then $ \rho $ is $ P_{ \theta } $-Anosov if and only if there is a $ \of{ \mathscr{ U } , \theta } $-admissible family $ Q = \of{ q_{ j , k } }_{ j , k } $ so that
\[
\log \of{ \frac{ \lambda_{ q_{ i , k } } \of{ \eta_{ i } \of{ \gamma } } }{ \lambda_{ q_{ j , k } + 1 } \of{ \eta_{ j } \of{ \gamma } } } }
\]
grows at least linearly in translation length $ \norm{ \gamma }_{ \Sigma } $ for all $ k \in \theta $ and $ 1 \leq i , j \leq m $ so that $ q_{ i , k } > 0 $ and $ q_{ j , k } < \dim_{ \K } \of{ U_{ j } } $. In this case, $ Q $ is the unique $ \of{ \mathscr{ U } , \theta } $-admissible family with the above property, and the following also hold:
\begin{enumerate}

\item[(1)]

The $ P_{ \theta } $-Anosov limit sets $ \xi_{ \rho }^{ \theta } \of{ \partial \Gamma } , \xi_{ \eta }^{ \theta } \of{ \partial \Gamma } \subseteq \Flag_{ \theta } \of{ V } $ consist entirely of weakly $ \of{ \mathscr{ U } , \theta , Q } $-structured and strongly $ \of{ \mathscr{ U } , \theta , Q } $-structured flags, respectively; that is,
\[
q_{ j , k } = \dim_{ \K } \of{ \pi_{ j } \of{ U_{ j }' \cap \xi_{ \rho }^{ k } \of{ z } } } = \dim_{ \K } \of{ U_{ j } \cap \xi_{ \eta }^{ k } \of{ z } }
\]
for all $ 1 \leq j \leq m $, $ k \in \theta $, and $ z \in \partial \Gamma $.

\item[(2)]

For each $ 1 \leq j \leq m $, let $ \theta_{ j } = \set{ q_{ j , k } }_{ k \in \theta } \cap \Delta_{ U_{ j } } $. The $ j $th block $ \eta_{ j } \colon \Gamma \to \GL \of{ U_{ j } } $ of $ \eta $ is $ P_{ \theta_{ j } } $-Anosov, with $ P_{ \theta_{ j } } $-Anosov limit map $ \xi_{ \eta_{ j } }^{ \theta_{ j } } \colon \partial \Gamma \to \Flag_{ \theta_{ j } } \of{ U_{ j } } $ given by the formulas
\[
\xi_{ \eta_{ j } }^{ \theta_{ j } } \of{ z } = \of{ \pi_{ j } \of{ U_{ j }' \cap \xi_{ \rho }^{ k_{ q } } \of{ z } } }_{ q \in \theta_{ j } } = \of{ U_{ j } \cap \xi_{ \eta }^{ k_{ q } } \of{ z } }_{ q \in \theta_{ j } } ,
\]
where for each $ q \in \theta_{ j } $, we choose any $ k_{ q } \in \Delta_{ V } $ so that $ q_{ j , k_{ q } } = q $.

\item[(3)]

The block diagonal representation $ \eta $ has $ P_{ \theta } $-Anosov limit map $ \xi_{ \eta }^{ \theta } \colon \partial \Gamma \to \Flag_{ \theta } \of{ V } $ given by the formulas
\[
\xi_{ \eta }^{ \theta } \of{ z } = \of{ \bigoplus_{ j = 1 }^{ m }{ \pi_{ j } \of{ U_{ j }' \cap \xi_{ \rho }^{ k } \of{ z } } } }_{ k \in \theta } = \of{ \bigoplus_{ j = 1 }^{ m }{ U_{ j } \cap \xi_{ \eta }^{ k } \of{ z } } }_{ k \in \theta } = \of{ \bigoplus_{ q_{ i , k } > 0 }{ \xi_{ \eta_{ i } }^{ q_{ i , k } } \of{ z } } }_{ k \in \theta } .
\]

\end{enumerate}

\end{theorem+}

\begin{proof}

Suppose first that $ \rho $ is $ P_{ \theta } $-Anosov, and recall that $ \rho $ is $ P_{ \theta } $-proximal, in that $ \rho \of{ \gamma } $ is $ P_{ \theta } $-biproximal for all infinite order $ \gamma \in \Gamma $. Fix one such infinite order $ \gamma_{ 0 } \in \Gamma $, and denote by $ Q = \of{ q_{ j , k } }_{ j , k } $ the large eigenvalue $ \theta $-configuration relative to $ \mathscr{ U } $; that is,
\[
q_{ j , k } = \dim_{ \K } \of{ \pi_{ j } \of{ U_{ j }' \cap \rho \of{ \gamma_{ 0 } }^{ + }_{ k } } } = \dim_{ \K } \of{ U_{ j } \cap \eta \of{ \gamma_{ 0 } }^{ + }_{ k } }
\]
for all $ 1 \leq j \leq m $ and $ k \in \theta $. Since $ \rho \of{ \gamma_{ 0 } } $ is $ P_{ \theta } $-proximal, (1) and (2) in \autoref{cor:reducibleProximality} implies that $ Q $ is $ \of{ \mathscr{ U } , \theta } $-admissible and that $ \rho \of{ \gamma_{ 0 } }^{ + }_{ \theta } \in W^{ \mathscr{ U } , \theta }_{ Q } \of{ V } $ and $ \eta \of{ \gamma_{ 0 } }^{ + }_{ \theta } \in S^{ \mathscr{ U } , \theta }_{ Q } \of{ V } $ are weakly and strongly $ \of{ \mathscr{ U } , \theta , Q } $-structured, respectively.

Recall that the $ P_{ \theta } $-Anosov limit maps $ \xi_{ \rho }^{ \theta } , \xi_{ \eta }^{ \theta } \colon \partial \Gamma \to \Flag_{ \theta } \of{ V } $ are continuous, $ \Gamma $-equivariant, and $ P_{ \theta } $-dynamics preserving, in the sense that $ \xi_{ \rho }^{ \theta } \of{ \delta^{ \pm } } = \rho \of{ \delta }^{ \pm }_{ \theta } $ and $ \xi_{ \eta }^{ \theta } \of{ \delta^{ \pm } } = \eta \of{ \delta }^{ \pm }_{ \theta } $ for all infinite order $ \delta \in \Gamma $. Since $ \rho $ preserves the flag associated to $ \mathscr{ U } $,
\begin{align*}
\dim_{ \K } \of{ \pi_{ j } \of{ U_{ j }' \cap \xi_{ \rho }^{ k } \of{ \gamma \cdot \gamma_{ 0 }^{ + } } } } & = \dim_{ \K } \of{ \pi_{ j } \of{ \rho \of{ \gamma } \of{ U_{ j }' } \cap \rho \of{ \gamma } \of{ \xi_{ \rho }^{ k } \of{ \gamma_{ 0 }^{ + } } } } } \\
& = \dim_{ \K } \of{ \pi_{ j } \of{ \rho \of{ \gamma } \of{ U_{ j }' \cap \rho \of{ \gamma_{ 0 } }^{ + }_{ k } } } } = \dim_{ \K } \of{ \pi_{ j } \of{ U_{ j }' \cap \rho \of{ \gamma_{ 0 } }^{ + }_{ k } } } = q_{ j , k } \\
\end{align*}
for all $ 1 \leq j \leq m $, $ k \in \theta $, and $ \gamma \in \Gamma $. Similarly, since $ \eta $ preserves $ \mathscr{ U } $ itself,
\begin{align*}
\dim_{ \K } \of{ U_{ j } \cap \xi_{ \eta }^{ k } \of{ \gamma \cdot \gamma_{ 0 }^{ + } } } & = \dim_{ \K } \of{ \eta \of{ \gamma } \of{ U_{ j } } \cap \eta \of{ \gamma } \of{ \xi_{ \eta }^{ k } \of{ \gamma_{ 0 }^{ + } } } } \\
& = \dim_{ \K } \of{ \eta \of{ \gamma } \of{ U_{ j } \cap \eta \of{ \gamma_{ 0 } }^{ + }_{ k } } } = \dim_{ \K } \of{ U_{ j } \cap \eta \of{ \gamma_{ 0 } }^{ + }_{ k } } = q_{ j , k }
\end{align*}
for all $ 1 \leq j \leq m $, $ k \in \theta $, and $ \gamma \in \Gamma $. Since $ \Gamma $ is a non-elementary word hyperbolic group, the orbit of any attracting ideal point is dense in the Gromov boundary $ \partial \Gamma $, and so
\begin{align*}
\xi_{ \rho }^{ \theta } \of{ \partial \Gamma } & = \xi_{ \rho }^{ \theta } \of{ \overline{ \Gamma \cdot \gamma_{ 0 }^{ + } } } \subseteq \overline{ \xi_{ \rho }^{ \theta } \of{ \Gamma \cdot \gamma_{ 0 }^{ + } } } \subseteq \overline{ W^{ U, \theta }_{ Q } \of{ V } } = W^{ \mathscr{ U } , \theta }_{ Q } \of{ V } \\
\xi_{ \eta }^{ \theta } \of{ \partial \Gamma } & = \xi_{ \eta }^{ \theta } \of{ \overline{ \Gamma \cdot \gamma_{ 0 }^{ + } } } \subseteq \overline{ \xi_{ \eta }^{ \theta } \of{ \Gamma \cdot \gamma_{ 0 }^{ + } } } \subseteq \overline{ S^{ U, \theta }_{ Q } \of{ V } } = S^{ \mathscr{ U } , \theta }_{ Q } \of{ V } .
\end{align*}
In particular, we conclude that $ \dim_{ \K } \of{ \pi_{ j } \of{ U_{ j }' \cap \xi_{ \rho }^{ k } \of{ z } } } = \dim_{ \K } \of{ U_{ j } \cap \xi_{ \eta }^{ k } \of{ z } } = q_{ j , k } $ for all $ 1 \leq j \leq m $, $ k \in \theta $, and $ z \in \partial \Gamma $. (3) in \autoref{cor:reducibleProximality} now implies that
\[
\log \of{ \frac{ \lambda_{ k } \of{ \rho \of{ \gamma } } }{ \lambda_{ k + 1 } \of{ \rho \of{ \gamma } } } } = \log \of{ \frac{ \lambda_{ k } \of{ \eta \of{ \gamma } } }{ \lambda_{ k + 1 } \of{ \eta \of{ \gamma } } } } = \min_{ i , j }{ \log \of{ \frac{ \lambda_{ q_{ i , k } } \of{ \eta_{ i } \of{ \gamma } } }{ \lambda_{ q_{ j , k } + 1 } \of{ \eta_{ j } \of{ \gamma } } } } }
\]
for all infinite order $ \gamma \in \Gamma $, where the above minimum is indexed by all $ 1 \leq i , j \leq m $ so that $ q_{ i , k } > 0 $ and $ q_{ j , k } < \dim_{ \K } \of{ U_{ j } } $. We note that the above equation holds for all finite order $ \gamma \in \Gamma $ as well; in this case, all eigenvalues of $ \rho \of{ \gamma } $ and $ \eta \of{ \gamma } $ are roots of unity, so that both sides of the above equation vanish. In particular,
\[
\log \of{ \frac{ \lambda_{ q_{ i , k } } \of{ \eta_{ i } \of{ \gamma } } }{ \lambda_{ q_{ j , k } + 1 } \of{ \eta_{ j } \of{ \gamma } } } } \geq \log \of{ \frac{ \lambda_{ k } \of{ \rho \of{ \gamma } } }{ \lambda_{ k + 1 } \of{ \rho \of{ \gamma } } } }
\]
grows at least linearly in translation length $ \norm{ \gamma }_{ \Sigma } $ for all $ 1 \leq i , j \leq m $ and $ k \in \theta $ so that $ q_{ i , k } > 0 $ and $ q_{ j , k } < \dim_{ \K } \of{ U_{ j } } $, as desired.

Conversely, now suppose that there is a $ \theta $-admissible family $ Q = \of{ q_{ j , k } }_{ j , k } $ so that
\[
\log \of{ \frac{ \lambda_{ q_{ i , k } } \of{ \eta_{ i } \of{ \gamma } } }{ \lambda_{ q_{ j , k } + 1 } \of{ \eta_{ j } \of{ \gamma } } } }
\]
grows at least linearly in translation length $ \norm{ \gamma }_{ \Sigma } $ for all $ 1 \leq i , j \leq m $ and $ k \in \theta $ so that $ q_{ i , k } > 0 $ and $ q_{ j , k } < \dim_{ \K } \of{ U_{ j } } $. For a given infinite order $ \gamma \in \Gamma $, note that for all $ 1 \leq i , j \leq m $ and $ k \in \theta $ so that $ q_{ i , k } > 0 $ and $ q_{ j , k } < \dim_{ \K } \of{ U_{ j } } $,
\begin{align*}
\lim_{ n \to \infty }{ n \log \of{ \frac{ \lambda_{ q_{ i , k } } \of{ \eta_{ i } \of{ \gamma } } }{ \lambda_{ q_{ j , k } + 1 } \of{ \eta_{ j } \of{ \gamma } } } } } & = \lim_{ n \to \infty }{ \log \of{ \of{ \frac{ \lambda_{ q_{ i , k } } \of{ \eta_{ i } \of{ \gamma } } }{ \lambda_{ q_{ j , k } + 1 } \of{ \eta_{ j } \of{ \gamma } } } }^{ n } } } = \lim_{ n \to \infty }{ \log \of{ \frac{ \lambda_{ q_{ i , j } } \of{ \eta_{ i } \of{ \gamma^{ n } } } }{ \lambda_{ q_{ j , k } + 1 } \of{ \eta_{ j } \of{ \gamma^{ n } } } } } } = \infty ,
\end{align*}
so that $ \lambda_{ q_{ i , k } } \of{ \eta_{ i } \of{ \gamma } } > \lambda_{ q_{ j , k } + 1 } \of{ \eta_{ j } \of{ \gamma } } $. We conclude that
\[
\log \of{ \frac{ \lambda_{ k } \of{ \rho \of{ \gamma } } }{ \lambda_{ k + 1 } \of{ \rho \of{ \gamma } } } } = \min_{ i , j }{ \log \of{ \frac{ \lambda_{ q_{ i , k } } \of{ \eta_{ i } \of{ \gamma } } }{ \lambda_{ q_{ j , k } + 1 } \of{ \eta_{ j } \of{ \gamma } } } } }
\]
for all such $ i $, $ j $, and $ k $. We again see that the above equation holds for all finite order $ \gamma \in \Gamma $ as well. Thus $ \log \of{ \frac{ \lambda_{ k } \of{ \rho \of{ \gamma } } }{ \lambda_{ k + 1 } \of{ \rho \of{ \gamma } } } } $ grows at least linearly in $ \norm{ \gamma }_{ \Sigma } $ for all $ k \in \theta $; that is, $ \rho $ is $ P_{ \theta } $-Anosov.

\begin{enumerate}

\item[(1)]

This was already proven in the above argument.

\item[(2)]

Note that
\[
\log \of{ \frac{ \lambda_{ q_{ j , k } } \of{ \eta_{ j } \of{ \gamma } } }{ \lambda_{ q_{ j , k } + 1 } \of{ \eta_{ j } \of{ \gamma } } } } \geq \log \of{ \frac{ \lambda_{ k } \of{ \rho \of{ \gamma } } }{ \lambda_{ k + 1 } \of{ \rho \of{ \gamma } } } }
\]
grows at least linearly in translation length $ \norm{ \gamma }_{ \Sigma } $ for all $ 1 \leq j \leq m $ and $ k \in \theta $ so that $ 0 < q_{ j , k } < \dim_{ \K } \of{ U_{ j } } $. Thus $ \eta_{ j } $ is $ P_{ \theta_{ j } } $-Anosov for all $ 1 \leq j \leq m $. The formulas
\begin{align*}
\xi_{ \eta_{ j } }^{ \theta_{ j } } \of{ z } & = \of{ \pi_{ j } \of{ U_{ j }' \cap \xi_{ \rho }^{ k_{ q } } \of{ z } } }_{ q \in \theta_{ j } } = \of{ U_{ j } \cap \xi_{ \eta }^{ k_{ q } } \of{ z } }_{ q \in \theta_{ j } }
\end{align*}
follow immediately from (1), specifically the observation that $ \xi_{ \eta }^{ \theta } \of{ z } $ is strongly $ \of{ \mathscr{ U } , \theta , Q } $-structured for all $ z \in \partial \Gamma $.

\item[(3)]

As in (2), the formulas
\[
\xi_{ \eta }^{ \theta } \of{ z } = \of{ \bigoplus_{ j = 1 }^{ m }{ \pi_{ j } \of{ U_{ j }' \cap \xi_{ \rho }^{ k } \of{ z } } } }_{ k \in \theta } = \of{ \bigoplus_{ j = 1 }^{ m }{ U_{ j } \cap \xi_{ \eta }^{ k } \of{ z } } }_{ k \in \theta } = \of{ \bigoplus_{ q_{ i , k } > 0 }{ \xi_{ \eta_{ i } }^{ q_{ i , k } } \of{ z } } }_{ k \in \theta }
\]
follow immediately from (1) and (2), specifically the observations that $ \xi_{ \rho }^{ \theta } \of{ z } $ and $ \xi_{ \eta }^{ \theta } \of{ z } $ are weakly and strongly $ \of{ \mathscr{ U } , \theta , Q } $-structured, respectively, for all $ z \in \partial \Gamma $. \qedhere

\end{enumerate}

\end{proof}

We remark that our standing hypothesis that $ \Gamma $ is non-elementary is necessary for \autoref{thm:eigenvalueConfiguration}. For a counterexample, we consider the elementary hyperbolic group $ \Z $, whose Gromov boundary is identified in the usual way with $ \set{ \pm \infty } $. Fix an integer $ 1 \leq k < \frac{ \dim_{ \K } \of{ V } }{ 2 } $ and a $ P_{ k } $-biproximal linear transformation $ \vect{ A } \in \GL \of{ V } $. Since $ \vect{ A } $ is $ P_{ k } $-biproximal, the representation $ \rho \colon \Z \to \GL \of{ V } $ defined by the formula $ \rho \of{ n } = \vect{ A }^{ n } $ is $ P_{ k } $-Anosov, with $ P_{ k } $-Anosov limit map $ \xi_{ \rho }^{ k } \colon \set{ \pm \infty } \to \Gr_{ k } \of{ V } $ given by $ \xi_{ \rho }^{ k } \of{ \pm \infty } = \vect{ A }^{ \pm }_{ k } $. Note that $ \rho $ preserves the direct sum decomposition $ V = \vect{ A }^{ + }_{ k } \oplus \vect{ A }^{ - }_{ \iota_{ V } \of{ k } } $, but
\begin{align*}
\dim_{ \K } \of{ \vect{ A }^{ + }_{ k } \cap \xi_{ \rho }^{ k } \of{ +\infty } } & = \dim_{ \K } \of{ \vect{ A }^{ + }_{ k } \cap \vect{ A }^{ + }_{ k } } = \dim_{ \K } \of{ \vect{ A }^{ + }_{ k } } = k \\\
& > 0 = \dim_{ \K } \of{ \set{ \vect{ 0 } } } = \dim_{ \K } \of{ \vect{ A }^{ + }_{ k } \cap \vect{ A }^{ - }_{ \iota_{ V } \of{ k } } } = \dim_{ \K } \of{ \vect{ A }^{ + }_{ k } \cap \xi_{ \rho }^{ k } \of{ -\infty } } .
\end{align*}
Thus $ \rho $ is $ P_{ k } $-Anosov but does not have consistent large eigenvalue $ \set{ k } $-configurations relative to the direct sum decomposition $ V = \vect{ A }^{ + }_{ k } \oplus \vect{ A }^{ - }_{ \iota_{ V } \of{ k } } $.

We end this section with a proof of a corollary of \autoref{thm:eigenvalueConfiguration}, the main content of which is that for a block upper triangular Anosov representation, no block can have too many or too few associated large eigenvalues.

\begin{corollary+} \label{cor:halfLargeEigenvalues}

Fix a subset $ \theta \subseteq \Delta_{ V } $ and consider a $ P_{ \theta } $-Anosov representation $ \rho \colon \Gamma \to \GL \of{ V } $ of $ \Gamma $ which is block upper triangular relative to $ \mathscr{ U } $, with large eigenvalue $ \theta $-configuration $ Q = \of{ q_{ j , k } }_{ j , k } $.
\begin{enumerate}

\item[(1)]

If $ k \in \theta $ and $ 1 \leq k \leq \frac{ \dim_{ \K } \of{ V } }{ 2 } $, then $ 0 \leq q_{ j , k } \leq \frac{ \dim_{ \K } \of{ U_{ j } } }{ 2 } $ for all $ 1 \leq j \leq m $.

\item[(2)]

If $ k \in \theta $ and $ \frac{ \dim_{ \K } \of{ V } }{ 2 } \leq k < \dim_{ \K } \of{ V } $, then $ \frac{ \dim_{ \K } \of{ U_{ j } } }{ 2 } \leq q_{ j , k } \leq \dim_{ \K } \of{ U_{ j } } $ for all $ 1 \leq j \leq m $.

\end{enumerate}

\begin{proof}

Denote by $ \eta = \bigoplus_{ j = 1 }^{ m }{ \eta_{ j } } $ the block diagonalization of $ \rho $ relative to $ \mathscr{ U } $, and recall that $ \lambda_{ k } \of{ \eta \of{ \gamma } } = \lambda_{ k } \of{ \rho \of{ \gamma } } $ for all $ 1 \leq k \leq \dim_{ \K } \of{ V } $ and $ \gamma \in \Gamma $. \autoref{thm:eigenvalueConfiguration} implies that for any infinite order $ \gamma \in \Gamma $,
\[
q_{ j , k } = \dim_{ \K } \of{ U_{ j }' \cap \rho \of{ \gamma }^{ + }_{ k } } = \dim_{ \K } \of{ U_{ j } \cap \eta \of{ \gamma }^{ + }_{ k } }
\]
is the number of complex eigenvalues (counted with algebraic multiplicity) of the $ j $th block $ \eta_{ j } \of{ \gamma } $ counted among the $ k $ largest eigenvalues of $ \rho \of{ \gamma } $.

\begin{enumerate}

\item[(1)]

By contraposition. Suppose that $ q_{ j , k } > \frac{ \dim_{ \K } \of{ U_{ j } } }{ 2 } $ for some $ 1 \leq j \leq m $ and $ k \in \theta $, so that for any infinite order $ \gamma \in \Gamma $, more than half of the eigenvalues of $ \eta_{ j } \of{ \gamma } $ are among the $ k $ largest eigenvalues of $ \rho \of{ \gamma } $. In particular, there is $ \ell \in \Delta_{ U_{ j } } $ so that both
\begin{align*}
\lambda_{ \ell } \of{ \eta_{ j } \of{ \gamma } } & \geq \lambda_{ k } \of{ \rho \of{ \gamma } } & \lambda_{ \iota_{ U_{ j } } \of{ \ell } + 1 } \of{ \eta_{ j } \of{ \gamma } } \geq \lambda_{ k } \of{ \rho \of{ \gamma } }
\end{align*}
for all infinite order $ \gamma \in \Gamma $. Fix one such infinite order $ \gamma \in \Gamma $, and consider $ 1 \leq i \leq m $ so that $ i \neq j $. Note that for each $ q \in \Delta_{ U_{ i } } $, if $ \lambda_{ q } \of{ \eta_{ i } \of{ \gamma } } < \lambda_{ k } \of{ \eta \of{ \gamma } } $, then \autoref{thm:eigenvalueConfiguration} implies that $ \lambda_{ q } \of{ \eta_{ i } \of{ \gamma^{ -1 } } } < \lambda_{ k } \of{ \eta \of{ \gamma^{ -1 } } } $, so that
\begin{align*}
\lambda_{ \iota_{ U_{ k } } \of{ q } + 1 } \of{ \eta_{ i } \of{ \gamma } } & = \frac{ 1 }{ \lambda_{ q } \of{ \eta_{ i } \of{ \gamma^{ -1 } } } } > \frac{ 1 }{ \lambda_{ i } \of{ \eta \of{ \gamma^{ -1 } } } } \\
& \geq \frac{ 1 }{ \lambda_{ \iota_{ U_{ j } } \of{ \ell } + 1 } \of{ \eta_{ j } \of{ \gamma^{ -1 } } } } = \lambda_{ \ell } \of{ \eta_{ j } \of{ \gamma } } \geq \lambda_{ k } \of{ \eta \of{ \gamma } } .
\end{align*}
So if $ \lambda_{ q } \of{ \eta_{ i } \of{ \gamma } } $ is among the $ \iota_{ V } \of{ k } $ smallest eigenvalue magnitudes of $ \eta \of{ \gamma } $, then $ \lambda_{ \iota_{ U_{ i } } \of{ q } + 1 } \of{ \eta_{ i } \of{ \gamma } } $ is among the $ k $ largest eigenvalue magnitudes of $ \eta \of{ \gamma } $. We conclude that at least half of the eigenvalues of $ \eta_{ i } \of{ \gamma } $ are among the $ k $ largest eigenvalues of $ \rho \of{ \gamma } $; that is, $ q_{ i , k } \geq \frac{ \dim_{ \K } \of{ U_{ i } } }{ 2 } $ for all $ 1 \leq i \leq m $ so that $ i \neq j $. Thus
\begin{align*}
k & = \sum_{ i = 1 }^{ m }{ q_{ i , k } } = q_{ j , k } + \sum_{ i \neq j }{ q_{ i , k } } > \frac{ \dim_{ \K } \of{ U_{ j } } }{ 2 } + \sum_{ i \neq j }{ \frac{ \dim_{ \K } \of{ U_{ i } } }{ 2 } } \\
& = \frac{ 1 }{ 2 } \sum_{ i = 1 }^{ m }{ \dim_{ \K } \of{ U_{ i } } } = \frac{ 1 }{ 2 } \dim_{ \K } \of{ \bigoplus_{ i = 1 }^{ m }{ U_{ i } } } = \frac{ \dim_{ \K } \of{ V } }{ 2 } .
\end{align*}

\item[(2)]

We note that $ \rho $ is $ P_{ \iota \of{ \theta } } $-Anosov and $ 1 \leq \iota_{ V } \of{ k } \leq \frac{ \dim_{ \K } \of{ V } }{ 2 } $, so that (1) above implies that $ 0 \leq q_{ j , \iota_{ V } \of{ k } } \leq \frac{ \dim_{ \K } \of{ U_{ j } } }{ 2 } $ for all $ 1 \leq j \leq m $. In particular,
\[
\dim_{ \K } \of{ U_{ j } } = \dim_{ \K } \of{ U_{ j } } - 0 \geq \dim_{ \K } \of{ U_{ j } } - q_{ j , k } \geq \dim_{ \K } \of{ U_{ j } } - \frac{ \dim_{ \K } \of{ U_{ j } } }{ 2 } = \frac{ \dim_{ \K } \of{ U_{ j } } }{ 2 }
\]
for all $ 1 \leq j \leq m $. The desired result now follows immediately from the observation that \autoref{thm:eigenvalueConfiguration} and \autoref{cor:reducibleProximality} imply that
\begin{align*}
q_{ j , k } + q_{ j , \iota_{ V } \of{ k } } & = \dim_{ \K } \of{ U_{ j } \cap \eta \of{ \gamma }^{ + }_{ k } } + \dim_{ \K } \of{ U_{ j } \cap \eta \of{ \gamma }^{ - }_{ \iota_{ V } \of{ k } } } \\
& = \dim_{ \K } \of{ \of{ U_{ j } \cap \eta \of{ \gamma }^{ + }_{ k } } \oplus \of{ U_{ j } \cap \eta \of{ \gamma }^{ - }_{ \iota_{ V } \of{ k } } } } = \dim_{ \K } \of{ U_{ j } }
\end{align*}
for all $ 1 \leq j \leq m $ and infinite order $ \gamma \in \Gamma $. \qedhere

\end{enumerate}

\end{proof}

\end{corollary+}

\section{Deformations of Reducible Anosov Representations}

Fix for this section a direct sum decomposition $ \mathscr{ U } = \of{ U_{ j } }_{ j = 1 }^{ m } $ of a non-zero finite-dimensional vector space $ V $ over $ \K $ with projection maps $ \pi_{ j } \colon V \onto U_{ j } $ and partial sums
\[
U_{ j }' = \bigoplus_{ i = 1 }^{ j }{ U_{ i } } ,
\]
for all $ 1 \leq j \leq m $, and also fix a representation $ \zeta \colon \Gamma \to N_{ \mathscr{ U } } $, which may or may not be completely reduced relative to $ \mathscr{ U } $. Recall that for each subset $ \theta \subseteq \Delta_{ V } $, we will denote by
\begin{align*}
A^{ \mathscr{ U } }_{ \theta } \of{ \zeta } & = \set{ \varphi \in \hom \of{ \Gamma , D_{ \mathscr{ U } } } : \beta_{ \mathscr{ U } } \of{ \delta , \varphi , \zeta } \textrm{ is $ P_{ \theta } $-Anosov for some (equivalently, any) } \delta \in \hom \of{ \Gamma , \R } } \\
& = \set{ \varphi \in \hom \of{ \Gamma , D_{ \mathscr{ U } } } : \beta_{ \mathscr{ U } } \of{ 0 , \varphi , \zeta } \textrm{ is $ P_{ \theta } $-Anosov} }
\end{align*}
the subset of the finite-dimensional real vector space $ \hom \of{ \Gamma , D_{ \mathscr{ U } } } $ corresponding to $ P_{ \theta } $-Anosov block diagonal representations with block normalization $ \zeta $. We note that $ A^{ \mathscr{ U } }_{ \theta } \of{ \zeta } $ is well-defined by (2) in \autoref{cor:blockDiagonalization}. The main goal of this section is to characterize $ A^{ \mathscr{ U } }_{ \theta } \of{ \zeta } $ and derive consequences about the deformation theory of reducible Anosov representations.

\subsection{Convex Deformation Spaces of Reducible Anosov Representations}

Before we prove \autoref{thm:domain}, we will require the following supporting lemma about block normalized linear transformations and strongly structured flags. Recall from the proof of \autoref{prop:blockDiagonalIsomorphism} that there is a group isomorphism $ \Psi_{ \mathscr{ U } } \colon R \times D_{ \mathscr{ U } } \times N_{ \mathscr{ U } } \to \Aut_{ \K } \of{ U } $ defined by the formula
\[
\Psi_{ \mathscr{ U } } \of{ s , \vect{ x } , \bigoplus_{ j = 1 }^{ m }{ \vect{ B }_{ j } } } = \bigoplus_{ j = 1 }^{ m }{ e^{ s + x_{ j } } \vect{ B }_{ j } } .
\]

\begin{lemma+} \label{lem:convexity}

For each subset $ \theta \subseteq \Delta_{ V } $, $ \of{ \mathscr{ U } , \theta } $-admissible family $ Q = \of{ q_{ j , k } }_{ j , k } $, and block normalized linear transformation $ \vect{ B } \in N_{ \mathscr{ U } } $, consider the subset
\[
C^{ \mathscr{ U } , \theta }_{ Q } \of{ \vect{ B } } = \set{ \of{ s , \vect{ x } } \in \R \times D_{ \mathscr{ U } } : \Psi_{ \mathscr{ U } } \of{ s , \vect{ x } , \vect{ B } } \textrm{ is $ P_{ \theta } $-proximal and } \Psi_{ \mathscr{ U } } \of{ s , \vect{ x } , \vect{ B } }^{ + }_{ \theta } \in S^{ \mathscr{ U } , \theta }_{ Q } \of{ V } } .
\]
of $ \R \times D_{ \mathscr{ U } } $.
\begin{enumerate}

\item[(1)]

If $ \of{ s , \vect{ x } } \in C^{ \mathscr{ U } , \theta }_{ Q } \of{ \vect{ B } } $, then $ \Psi_{ \mathscr{ U } } \of{ s , \vect{ x } , \vect{ B } } $ has logarithmic eigenvalue gaps
\[
\log \of{ \frac{ \lambda_{ k } \of{ \Psi_{ \mathscr{ U } } \of{ s , \vect{ x } , \vect{ B } } } }{ \lambda_{ k + 1 } \of{ \Psi_{ \mathscr{ U } } \of{ s , \vect{ x } , \vect{ B } } } } } = \min_{ i , j }{ \log \of{ \frac{ \lambda_{ q_{ i , k } } \of{ \vect{ B }_{ i } } }{ \lambda_{ q_{ j , k } + 1 } \of{ \vect{ B }_{ j } } } } + x_{ i } - x_{ j } }
\]
for all $ k \in \theta $, where $ \vect{ B } $ has block decomposition $ \vect{ B } = \bigoplus_{ j = 1 }^{ m }{ \vect{ B }_{ j } } $ relative to $ \mathscr{ U } $ and the above minimum is indexed by all $ 1 \leq i , j \leq m $ so that $ q_{ i , k } > 0 $ and $ q_{ j , k } < \dim_{ \K } \of{ U_{ j } } $.

\item[(2)]

$ C^{ \mathscr{ U } , \theta }_{ Q } \of{ \vect{ B } } $ is convex in $ \R \times D_{ \mathscr{ U } } $.

\end{enumerate}

\begin{proof}

\begin{enumerate}

\item[(1)]

Since $ \Psi_{ \mathscr{ U } } \of{ s , \vect{ x } , \vect{ B } } $ is $ P_{ \theta } $-proximal and has block decomposition $ \Psi_{ \mathscr{ U } } \of{ s , \vect{ x } , \vect{ B } } = \bigoplus_{ j = 1 }^{ m }{ e^{ s + x_{ j } } \vect{ B }_{ j } } $ relative to $ \mathscr{ U } $, (3) in \autoref{cor:reducibleProximality} implies that it has logarithmic eigenvalue gaps
\begin{align*}
\log \of{ \frac{ \lambda_{ k } \of{ \Psi_{ \mathscr{ U } } \of{ s , \vect{ x } , \vect{ B } } } }{ \lambda_{ k + 1 } \of{ \Psi_{ \mathscr{ U } } \of{ s , \vect{ x } , \vect{ B } } } } } & = \min_{ i , j }{ \log \of{ \frac{ \lambda_{ q_{ i , k } } \of{ e^{ s + x_{ i } } \vect{ B }_{ i } } }{ \lambda_{ q_{ j , k } + 1 } \of{ e^{ s + x_{ j } } \vect{ B }_{ j } } } } } = \min_{ i , j }{ \log \of{ \frac{ e^{ x_{ i } } \lambda_{ q_{ i , k } } \of{ \vect{ B }_{ i } } }{ e^{ x_{ j } } \lambda_{ q_{ j , k } + 1 } \of{ \vect{ B }_{ j } } } } } \\
& = \min_{ j , k }{ \log \of{ \frac{ \lambda_{ q_{ i , k } } \of{ \vect{ B }_{ i } } }{ \lambda_{ q_{ j , k } + 1 } \of{ \vect{ B }_{ j } } } } } + x_{ i } - x_{ j }
\end{align*}
for all $ k \in \theta $, where the above minima are indexed by all $ 1 \leq i , j \leq m $ so that $ q_{ i , k } > 0 $ and $ q_{ j , k } < \dim_{ \K } \of{ U_{ j } } $.

\item[(2)]

Consider two points $ \of{ s^{ 0 } , \vect{ x }^{ 0 } } , \of{ s^{ 1 } , \vect{ x }^{ 1 } } \in C^{ \mathscr{ U } , \theta }_{ Q } \of{ \vect{ B } } $. For each $ 0 \leq t \leq 1 $, let
\begin{align*}
s^{ t } & = \of{ 1 - t } s^{ 0 } + t s^{ 1 } & \vect{ x }^{ t } & = \of{ 1 - t } \vect{ x }^{ 0 } + t \vect{ x }^{ 1 } ,
\end{align*}
and note that $ \Psi_{ \mathscr{ U } } \of{ s^{ t } , \vect{ x }^{ t } , \vect{ B } } $ has block decomposition
\[
\Psi_{ \mathscr{ U } } \of{ s^{ t } , \vect{ x }^{ t } , \vect{ B } } = \bigoplus_{ j = 1 }^{ m }{ e^{ s^{ t } + x_{ j }^{ t } } \vect{ B }_{ j } }
\]
relative to $ \mathscr{ U } $. Note that for each $ 1 \leq i , j \leq m $ and $ k \in \theta $ so that $ q_{ i , k } > 0 $ and $ q_{ j , k } < \dim_{ \K } \of{ U_{ j } } $,
\[
0 < \log \of{ \frac{ \lambda_{ k } \of{ \Psi_{ \mathscr{ U } } \of{ s^{ t } , \vect{ x }^{ t } , \vect{ B } } } }{ \lambda_{ k + 1 } \of{ \Psi_{ \mathscr{ U } } \of{ s^{ t } , \vect{ x }^{ t } , \vect{ B } } } } } \leq \log \of{ \frac{ \lambda_{ q_{ i , k } } \of{ \vect{ B }_{ i } } }{ \lambda_{ q_{ j , k } + 1 } \of{ \vect{ B }_{ j } } } } + x_{ i }^{ t } - x_{ j }^{ t }
\]
for $ t \in \set{ 0 , 1 } $ by (1) above. The righthand side of the above inequality is a linear (and therefore monotone) function of $ t $, and so
\[
0 < \log \of{ \frac{ \lambda_{ q_{ i , k } } \of{ \vect{ B }_{ i } } }{ \lambda_{ q_{ j , k } + 1 } \of{ \vect{ B }_{ j } } } } + x_{ i }^{ t } - x_{ j }^{ t } = \log \of{ \frac{ \lambda_{ q_{ i , k } } \of{ e^{ s^{ t } + x_{ j }^{ t } } \vect{ B }_{ i } } }{ \lambda_{ q_{ j , k } + 1 } \of{ e^{ s^{ t } + x_{ j }^{ t } } \vect{ B }_{ j } } } }
\]
for all $ 0 \leq t \leq 1 $. We conclude that
\[
\min_{ i }{ \lambda_{ q_{ i , k } } \of{ e^{ s^{ t } + x_{ i }^{ t } } \vect{ B }_{ i } } } > \max_{ j }{ \lambda_{ q_{ j , k } + 1 } \of{ e^{ s^{ t } + x_{ j }^{ t } } \vect{ B }_{ j } } }
\]
for all $ k \in \theta $ and $ 0 \leq t \leq 1 $, where the above minimum and maximum are indexed by all $ 1 \leq i , j \leq m $ so that $ q_{ i , k } > 0 $ and $ q_{ j , k } < \dim_{ \K } \of{ U_{ j } } $. Since $ Q $ is $ \of{ \mathscr{ U } , \theta } $-admissible, this implies that $ \Psi_{ \mathscr{ U } } \of{ s^{ t } , \vect{ x }^{ t } , \vect{ B } } $ is $ P_{ \theta } $-proximal for all $ 0 \leq t \leq 1 $, and moreover that
\[
\dim_{ \K } \of{ U_{ j } \cap \Psi_{ \mathscr{ U } } \of{ s^{ t } , \vect{ x }^{ t } , \vect{ B }^{ + }_{ k } } } = \dim_{ \K } \of{ \of{ e^{ s^{ t } + x_{ j }^{ t } } \vect{ B }_{ j } }^{ + }_{ q_{ j , k } } } = q_{ j , k }
\]
for all $ 1 \leq j \leq m $ and $ k \in \theta $. Thus $ \Psi_{ \mathscr{ U } } \of{ s^{ t } , \vect{ x }^{ t } , \vect{ B } }^{ + }_{ \theta } \in S^{ \mathscr{ U } , \theta }_{ Q } \of{ V } $ is strongly $ \of{ \mathscr{ U } , \theta , Q } $-structured, and so $ \of{ s^{ t } , \vect{ x }^{ t } } \in C^{ \mathscr{ U } , \theta }_{ Q } \of{ \vect{ B } } $. Since this holds for all $ 0 \leq t \leq 1 $, $ C^{ \mathscr{ U } , \theta }_{ Q } \of{ \vect{ B } } $ is convex in $ \R \times D_{ \mathscr{ U } } $. \qedhere

\end{enumerate}

\end{proof}

\end{lemma+}

We are now ready to use the tools developed above to prove \autoref{thm:domain}, which we restate below for the reader's convenience:

\domain*

\begin{proof}

That $ A^{ \mathscr{ U } }_{ \theta } \of{ \zeta } $ is open in $ \hom \of{ \Gamma , D_{ \mathscr{ U } } } $ follows immediately from \autoref{thm:stability}, the stability of the $ P_{ \theta } $-Anosov condition. It remains to show convexity and boundedness.

Fix an infinite order element $ \gamma_{ 0 } \in \comm{ \Gamma }{ \Gamma } $ of the commutator subgroup $ \comm{ \Gamma }{ \Gamma } \leq \Gamma $, and consider the family $ Q = \of{ q_{ j , k } }_{ j , k } $ indexed by $ 1 \leq j \leq m $ and $ k \in \theta $ and given by the formula
\[
q_{ j , k } = \dim_{ \K } \of{ U_{ j } \cap \zeta \of{ \gamma_{ 0 } } }^{ + }_{ k } .
\]
Note that since $ \R $ and $ D_{ \mathscr{ U } } $ are abelian, $ \gamma_{ 0 } \in \ker \of{ \delta } \cap \ker \of{ \varphi } $ for all homomorphisms $ \delta \colon \Gamma \to \R $ and $ \varphi \colon \Gamma \to D_{ \mathscr{ U } } $. We see that
\[
\beta_{ \mathscr{ U } } \of{ \delta , \varphi , \zeta } \of{ \gamma_{ 0 } } = \Psi_{ \mathscr{ U } } \of{ \delta \of{ \gamma_{ 0 } } , \varphi \of{ \gamma_{ 0 } } , \zeta \of{ \gamma_{ 0 } } } = \Psi_{ \mathscr{ U } } \of{ 0 , \vect{ 0 } , \zeta \of{ \gamma_{ 0 } } } = \bigoplus_{ j = 1 }^{ m }{ e^{ 0 + 0 } \zeta_{ j } \of{ \gamma_{ 0 } } } = \bigoplus_{ j = 1 }^{ m }{ \zeta_{ j } \of{ \gamma_{ 0 } } } = \zeta \of{ \gamma_{ 0 } }
\]
for all homomorphisms $ \delta \colon \Gamma \to \R $ and $ \varphi \colon \Gamma \to D_{ \mathscr{ U } } $. This implies that $ Q $ is the large eigenvalue $ \theta $-configuration relative to $ \mathscr{ U } $ of every $ P_{ \theta } $-Anosov representation of the form $ \beta_{ \mathscr{ U } } \of{ \delta , \varphi , \zeta } \colon \Gamma \to \Aut_{ \K } \of{ \mathscr{ U } } $.

For convexity, fix $ \varphi^{ 0 } , \varphi^{ 1 } \in A^{ \mathscr{ U } }_{ \theta } \of{ \zeta } $, so that the block diagonal representations $ \rho^{ 0 } = \beta_{ \mathscr{ U } } \of{ \delta^{ 0 } , \varphi^{ 0 } , \zeta } $ and $ \rho^{ 1 } = \beta_{ \mathscr{ U } } \of{ \delta^{ 1 } , \varphi^{ 1 } , \zeta } $ are $ P_{ \theta } $-Anosov for some homomorphisms $ \delta^{ 0 } , \delta^{ 1 } \colon \Gamma \to \R $. Then for each $ k \in \theta $, there are constants $ a_{ k }^{ 0 } , a_{ k }^{ 1 } > 0 $ and $ b_{ k }^{ 0 } , b_{ k }^{ 1 } \geq 0 $ so that
\begin{align*}
\log \of{ \frac{ \lambda_{ k } \of{ \rho^{ 0 } \of{ \gamma } } }{ \lambda_{ k + 1 } \of{ \rho^{ 0 } \of{ \gamma } } } } & \geq a_{ k }^{ 0 } \norm{ \gamma }_{ \Sigma } - b_{ k }^{ 0 } & \log \of{ \frac{ \lambda_{ k } \of{ \rho^{ 1 } \of{ \gamma } } }{ \lambda_{ k + 1 } \of{ \rho^{ 1 } \of{ \gamma } } } } & \geq a_{ k }^{ 1 } \norm{ \gamma }_{ \Sigma } - b_{ k }^{ 1 }
\end{align*}
for all $ \gamma \in \Gamma $. Let
\begin{align*}
\varphi^{ t } & = \of{ 1 - t } \varphi^{ 0 } + t \varphi^{ 1 } & a_{ k }^{ t } & = \of{ 1 - t } a_{ k }^{ 0 } + t a_{ k }^{ 1 } & \delta^{ t } & = \of{ 1 - t } \delta^{ 0 } + t \delta^{ 1 } \\
\rho^{ t } & = \beta_{ \mathscr{ U } } \of{ \delta^{ t } , \varphi^{ t } , \zeta } & b_{ k }^{ t } & = \of{ 1 - t } b_{ k }^{ 0 } + t b_{ k }^{ 1 }
\end{align*}
for all $ 0 < t < 1 $ and $ k \in \theta $. Note that \autoref{thm:eigenvalueConfiguration} implies that $ Q $ is $ \of{ \mathscr{ U } , \theta } $-admissible and that $ \of{ \delta^{ t } \of{ \gamma } , \varphi^{ t } \of{ \gamma } } \in C^{ \mathscr{ U } , \theta }_{ Q } \of{ \eta \of{ \gamma } } $ for all $ t \in \set{ 0 , 1 } $ and infinite order $ \gamma \in \Gamma $. We conclude by (2) in \autoref{lem:convexity} that $ \of{ \delta^{ t } \of{ \gamma } , \varphi^{ t } \of{ \gamma } } \in C^{ \mathscr{ U } , \theta }_{ Q } \of{ \eta \of{ \gamma } } $ for all $ 0 \leq t \leq 1 $ and infinite order $ \gamma \in \Gamma $. (1) in \autoref{lem:convexity} now implies that
\[
\log \of{ \frac{ \lambda_{ k } \of{ \rho^{ t } \of{ \gamma } } }{ \lambda_{ k + 1 } \of{ \rho^{ t } \of{ \gamma } } } } = \min_{ i , j }{ \log \of{ \frac{ \lambda_{ q_{ i , k } } \of{ \zeta_{ i } \of{ \gamma } } }{ \lambda_{ q_{ j , k } + 1 } \of{ \zeta_{ j } \of{ \gamma } } } } + \varphi_{ i }^{ t } \of{ \gamma } - \varphi_{ j }^{ t } \of{ \gamma } }
\]
for all $ k \in \theta $, $ 0 \leq t \leq 1 $, and infinite order $ \gamma \in \Gamma $, where the above minimum is indexed by all $ 1 \leq i , j \leq m $ so that $ q_{ i , k } > 0 $ and $ q_{ j , k } < \dim_{ \K } \of{ U_{ j } } $. We observe that the above also holds for finite order $ \gamma \in \Gamma $, since both sides of the above equation vanish. Note that for all $ 0 \leq t \leq 1 $,
\begin{align*}
\log \of{ \frac{ \lambda_{ q_{ i , k } } \of{ \zeta_{ i } \of{ \gamma } } }{ \lambda_{ q_{ j , k } + 1 } \of{ \zeta_{ j } \of{ \gamma } } } } + \varphi_{ i }^{ t } \of{ \gamma } - \varphi_{ j }^{ t } \of{ \gamma } & = \of{ 1 - t } \of{ \log \of{ \frac{ \lambda_{ q_{ i , k } } \of{ \zeta_{ i } \of{ \gamma } } }{ \lambda_{ q_{ j , k } + 1 } \of{ \zeta_{ j } \of{ \gamma } } } } + \varphi_{ i }^{ 0 } \of{ \gamma } - \varphi_{ j }^{ 0 } \of{ \gamma } } \\
& \qquad + t \of{ \log \of{ \frac{ \lambda_{ q_{ i , k } } \of{ \zeta_{ i } \of{ \gamma } } }{ \lambda_{ q_{ j , k } + 1 } \of{ \zeta_{ j } \of{ \gamma } } } } + \varphi_{ i }^{ 1 } \of{ \gamma } - \varphi_{ j }^{ 1 } \of{ \gamma } } \\
& \geq \of{ 1 - t } \log \of{ \frac{ \lambda_{ k } \of{ \rho^{ 0 } \of{ \gamma } } }{ \lambda_{ k + 1 } \of{ \rho^{ 0 } \of{ \gamma } } } } + t \log \of{ \frac{ \lambda_{ k } \of{ \rho^{ 1 } \of{ \gamma } } }{ \lambda_{ k + 1 } \of{ \rho^{ 1 } \of{ \gamma } } } } \\
& \geq \of{ 1 - t } \of{ a_{ k }^{ 0 } \norm{ \gamma }_{ \Sigma } - b_{ k }^{ 0 } } + t \of{ a_{ k }^{ 1 } \norm{ \gamma }_{ \Sigma } - b_{ k }^{ 1 } } = a_{ k }^{ t } \norm{ \gamma }_{ \Sigma } - b_{ k }^{ t }
\end{align*}
grows at least linearly in translation length $ \norm{ \gamma }_{ \Sigma } $ for all $ 1 \leq i , j \leq m $ and $ k \in \theta $ so that $ q_{ i , k } > 0 $ and $ q_{ j , k } < \dim_{ \K } \of{ U_{ j } } $. This implies that
\[
\log \of{ \frac{ \lambda_{ k } \of{ \rho^{ t } \of{ \gamma } } }{ \lambda_{ k + 1 } \of{ \rho^{ t } \of{ \gamma } } } } = \min_{ i , j }{ \log \of{ \frac{ \lambda_{ q_{ i , k } } \of{ \zeta_{ i } \of{ \gamma } } }{ \lambda_{ q_{ j , k } + 1 } \of{ \zeta_{ j } \of{ \gamma } } } } + \varphi_{ i }^{ t } \of{ \gamma } - \varphi_{ j }^{ t } \of{ \gamma } }
\]
grows at least linearly in $ \norm{ \gamma }_{ \Sigma } $ for all $ k \in \theta $ and $ 0 \leq t \leq 1 $, so that $ \rho^{ t } $ is $ P_{ \theta } $-Anosov. We conclude that $ A^{ \mathscr{ U } }_{ \theta } \of{ \zeta } $ is convex in $ \hom \of{ \Gamma , D_{ \mathscr{ U } } } $.

For boundedness, we first note that if $ m = 1 $, then $ \hom \of{ \Gamma , D_{ \mathscr{ U } } } = \set{ 0 } $ and so every subset is bounded. Thus we may assume that $ m \geq 2 $. Moreover,
\[
A^{ \mathscr{ U } }_{ \theta } \of{ \zeta } = \bigcap_{ k \in \theta }{ A^{ \mathscr{ U } }_{ \set{ k } } \of{ \zeta } } = \bigcap_{ \substack{ k \in \theta \\ 1 \leq k \leq \frac{ \dim_{ \K } \of{ V } }{ 2 } } }{ A^{ \mathscr{ U } }_{ \set{ k } } \of{ \zeta } }
\]
by \autoref{fact:AnosovInvolution}, and so it suffices to show that $ A^{ \mathscr{ U } }_{ \set{ k } } \of{ \zeta } $ is bounded for all $ k \in \Delta_{ V } $ so that $ 1 \leq k \leq \frac{ \dim_{ \K } \of{ V } }{ 2 } $.

To that end, suppose for a contradiction that $ A^{ \mathscr{ U } }_{ \set{ k } } \of{ \zeta } $ is unbounded for some $ k \in \Delta_{ V } $ so that $ 1 \leq k \leq \frac{ \dim_{ \K } \of{ V } }{ 2 } $. As an unbounded convex set in a finite-dimensional real vector space, it contains a ray $ \set{ \varphi + t \psi \in \hom \of{ \Gamma , D_{ \mathscr{ U } } } : t \geq 0 } \subseteq A^{ \mathscr{ U } }_{ \set{ k } } \of{ \zeta } $. For each $ t \geq 0 $, let $ \rho^{ t } = \beta_{ \mathscr{ U } } \of{ 0 , \varphi + t \psi , \zeta } $, and note that $ \rho^{ t } $ has large eigenvalue $ \theta $-configuration $ Q $ relative to $ \mathscr{ U } $. \autoref{cor:halfLargeEigenvalues} implies that
\[
0 \leq q_{ j , k } \leq \frac{ \dim_{ \K } \of{ U_{ j } } }{ 2 } < \dim_{ \K } \of{ U_{ j } }
\]
for all $ 1 \leq j \leq m $. Fix $ 1 \leq i \leq m $ so that $ q_{ i , k } > 0 $; such an $ i $ exists by \autoref{lem:nonEmptyAdmissibleCondition}. If $ \psi_{ i } \of{ \gamma } = \psi_{ j } \of{ \gamma } $ for all $ 1 \leq j \leq m $ and $ \gamma \in \Gamma $, then
\[
0 = \sum_{ j = 1 }^{ m }{ \dim_{ \K } \of{ U_{ j } } \psi_{ j } \of{ \gamma } } = \sum_{ j = 1 }^{ m }{ \dim_{ \K } \of{ U_{ j } } \psi_{ i } \of{ \gamma } } = \dim_{ \K } \of{ \bigoplus_{ j = 1 }^{ m }{ U_{ j } } } \psi_{ i } \of{ \gamma } = \dim_{ \K } \of{ V } \psi_{ i } \of{ \gamma }
\]
for all $ \gamma \in \Gamma $, so that $ \psi = 0 $. This contradiction implies that $ \psi_{ i } \of{ \gamma } < \psi_{ j } \of{ \gamma } $ for some $ 1 \leq j \leq m $ and $ \gamma \in \Gamma $, which necessarily has infinite order in $ \Gamma $ because $ \R $ is torsion-free. We now observe that \autoref{thm:eigenvalueConfiguration} implies that
\[
\lambda_{ k } \of{ \rho^{ t } \of{ \gamma } } \leq e^{ \varphi_{ i } \of{ \gamma } + t \psi_{ i } \of{ \gamma } } \lambda_{ q_{ i , k } } \of{ \zeta_{ i } \of{ \gamma } } \leq e^{ \varphi_{ j } \of{ \gamma } + t \psi_{ j } \of{ \gamma } } \lambda_{ q_{ j , k } + 1 } \of{ \zeta_{ j } \of{ \gamma } } \leq \lambda_{ k + 1 } \of{ \rho^{ t } \of{ \gamma } }
\]
whenever
\[
t \geq \frac{ 1 }{ \psi_{ j } \of{ \gamma } - \psi_{ i } \of{ \gamma } } \of{ \log \of{ \frac{ \lambda_{ q_{ i , k } } \of{ \zeta_{ i } \of{ \gamma } } }{ \lambda_{ q_{ j , k } + 1 } \of{ \zeta_{ j } \of{ \gamma } } } } + \varphi_{ i } \of{ \gamma } - \varphi_{ j } \of{ \gamma } } .
\]
This contradiction implies that $ A^{ \mathscr{ U } }_{ \set{ k } } \of{ \zeta } $ is bounded. \qedhere

\end{proof}

One important consequence of \autoref{thm:domain} is that for many word hyperbolic groups, components of the character variety containing (equivalence classes of) reducible representations must also contain (equivalence classes of) non-Anosov representations. The following lemma details the conditions on $ \Gamma $ for which these conclusions apply.

\begin{lemma+} \label{lem:commutatorSubgroupHomomorphisms}

The commutator subgroup $ \comm{ \Gamma }{ \Gamma } $ has finite index in $ \Gamma $ if and only if $ \hom \of{ \Gamma , D_{ \mathscr{ U } } } = \set{ 0 } $.

\begin{proof}

First suppose that $ \comm{ \Gamma }{ \Gamma } $ has finite index in $ \Gamma $, and consider a block deformation $ \Gamma \to D_{ \mathscr{ U } } $. Since $ D_{ \mathscr{ U } } $ is abelian, this block deformation factors through a homomorphism $ \faktor{ \Gamma }{ \comm{ \Gamma }{ \Gamma } } $. However, since $ \faktor{ \Gamma }{ \comm{ \Gamma }{ \Gamma } } $ is finite and $ D_{ \mathscr{ U } } $ is torsion-free, there are no non-zero homomorphisms $ \faktor{ \Gamma }{ \comm{ \Gamma }{ \Gamma } } \to D_{ \mathscr{ U } } $. Thus there are no non-zero block deformations $ \Gamma \to D_{ \mathscr{ U } } $.

On the other hand, suppose that $ \comm{ \Gamma }{ \Gamma } $ has infinite index in $ \Gamma $. Since $ \Gamma $ is finitely generated, its abelianization $ \faktor{ \Gamma }{ \comm{ \Gamma }{ \Gamma } } $ is a finitely generated infinite abelian group, and therefore contains an infinite cyclic group $ C $ as a direct summand. Since $ \dim_{ \K } \of{ V } > 0 $, $ \dim_{ \K } \of{ D_{ \mathscr{ U } } } > 0 $, and we may choose a nonzero $ \vect{ x } \in D_{ \mathscr{ U } } $. The composition
\[
\begin{tikzcd}
\Gamma \ar[ r , two heads ] & \faktor{ \Gamma }{ \comm{ \Gamma }{ \Gamma } } \ar[ r , two heads ] & C \ar[ r , two heads ] & \Z \vect{ x } \ar[ r , hook ] & D_{ \mathscr{ U } }
\end{tikzcd}
\]
is a non-zero block deformation $ \Gamma \to D_{ \mathscr{ U } } $. \qedhere

\end{proof}

\end{lemma+}

\teichmuller*

\begin{proof}

By contraposition. Fix a connected component $ C $ of $ \Char_{ H } \of{ \Gamma , G } $ which contains the $ H $-equivalence class $ \class{ \rho }_{ H } $ of a reducible representation $ \rho \colon \Gamma \to G $. Let $ \mathscr{ U } $ be a non-trivial direct sum decomposition of $ V $ whose associated flag is preserved by $ \rho $. \autoref{cor:blockDiagonalization} implies that $ \rho $ is $ H $-equivalent to a block diagonal representation $ \beta_{ \mathscr{ U } } \of{ \delta , \varphi , \zeta } $ relative to $ \mathscr{ U } $.

Consider the continuous map $ D \colon \hom \of{ \Gamma , D_{ \mathscr{ U } } } \to \Char_{ H } \of{ \Gamma , G } $ defined by the formula $ D \of{ \psi } = \class{ \beta_{ \mathscr{ U } } \of{ \delta , \psi , \zeta } }_{ H } $, and denote by $ A^{ \mathscr{ U } } \of{ \zeta } $ the subset of $ \hom \of{ \Gamma , D_{ \mathscr{ U } } } $ corresponding to block diagonal Anosov representations with block normalization $ \zeta $ relative to $ \mathscr{ U } $; that is,
\[
A^{ \mathscr{ U } } \of{ \zeta } = \bigcup_{ \emptyset \subsetneq \theta \subseteq \Delta_{ V } }{ A^{ \mathscr{ U } }_{ \theta } \of{ \zeta } } .
\]
By \autoref{thm:domain}, $ A^{ \mathscr{ U } } \of{ \zeta } $ is a bounded subset of the finite-dimensional real vector space $ \hom \of{ \Gamma , D_{ \mathscr{ U } } } $. \autoref{lem:commutatorSubgroupHomomorphisms} implies that $ \hom \of{ \Gamma , D_{ \mathscr{ U } } } $ is a non-zero vector space, so that this bounded subset $ A^{ \mathscr{ U } } \of{ \zeta } $ is not the entirety of $ \hom \of{ \Gamma , D_{ \mathscr{ U } } } $; that is, there is some $ \psi \in \hom \of{ \Gamma , D_{ \mathscr{ U } } } \setminus A^{ \mathscr{ U } } \of{ \zeta } $. We note that $ D \of{ \psi } \in C $ is not Anosov. \qedhere

\end{proof}

We remark that the hypothesis that $ \comm{ \Gamma }{ \Gamma } $ has infinite index in $ \Gamma $ is necessary in order for the above argument to work. If $ \comm{ \Gamma }{ \Gamma } $ has finite index in $ \Gamma $, then \autoref{lem:commutatorSubgroupHomomorphisms} implies that there are no non-zero homomorphisms $ \Gamma \to D_{ \mathscr{ U } } $, and these homomorphisms parametrize the relevant block deformations of the block normalized representation $ \zeta \colon \Gamma \to N_{ \mathscr{ U } } $.

\subsection{Another Characterizations of the Anosov Condition}

For the remainder of this section, we aim to use the tools developed above to extend the characterization of the Anosov condition for reducible representations presented in \autoref{thm:eigenvalueConfiguration} to some class of reducible representations. Specifically, our condition will hold for block upper triangular representations for which the block normalization of the block diagonalization is itself Anosov.

\AnosovCharacterization*

\begin{proof}

Fix homomorphisms $ \delta \colon \Gamma \to \R $ and $ \varphi \colon \Gamma \to D_{ \mathscr{ U } } $, and consider the block diagonal representation $ \eta = \beta_{ \mathscr{ U } } \of{ \delta , \varphi , \zeta } \colon \Gamma \to \Aut_{ \K } \of{ \mathscr{ U } } $. Let $ s $ denote the above supremum; that is,
\[
s = \sup_{ i , j , k , \gamma }{ \frac{ \varphi_{ j } \of{ \gamma } - \varphi_{ i } \of{ \gamma } }{ \log \of{ \frac{ \lambda_{ q_{ i , k } } \of{ \zeta_{ i } \of{ \gamma } } }{ \lambda_{ q_{ j , k } + 1 } \of{ \zeta_{ j } \of{ \gamma } } } } } } ,
\]
where this supremum is indexed by all $ 1 \leq i , j \leq m $ and $ k \in \theta $ so that $ q_{ i , k } > 0 $ and $ q_{ j , k } < \dim_{ \K } \of{ U_{ j } } $, and all infinite order $ \gamma \in \Gamma $. It suffices to show that $ \eta $ is $ P_{ \theta } $-Anosov if and only if $ s < 1 $.

To that end, we note that since $ \zeta $ is $ P_{ \theta } $-Anosov with large eigenvalue $ \theta $-configuration $ Q $ relative to $ \mathscr{ U } $, \autoref{thm:eigenvalueConfiguration} implies that $ \lambda_{ q_{ i , j } } \of{ \zeta_{ i } \of{ \gamma } } > \lambda_{ q_{ j , k } + 1 } \of{ \zeta_{ j } \of{ \gamma } } $ for all $ 1 \leq i , j \leq m $ and $ k \in \theta $ so that $ q_{ i , k } > 0 $ and $ q_{ j , k } < \dim_{ \K } \of{ U_{ j } } $ and all infinite order $ \gamma \in \Gamma $. We conclude that
\[
s \log \of{ \frac{ \lambda_{ q_{ i , k } } \of{ \zeta_{ i } \of{ \gamma } } }{ \lambda_{ q_{ j , k } + 1 } \of{ \zeta_{ j } \of{ \gamma } } } } \geq \varphi_{ j } \of{ \gamma } - \varphi_{ i } \of{ \gamma }
\]
for all such $ i $, $ j $, $ k $, and $ \gamma $.

First suppose that $ s < 1 $, and note that
\begin{align*}
\log \of{ \frac{ \lambda_{ q_{ i , k } } \of{ \eta_{ i } \of{ \gamma } } }{ \lambda_{ q_{ j , k } + 1 } \of{ \eta_{ j } \of{ \gamma } } } } & = \log \of{ \frac{ e^{ \delta \of{ \gamma } + \varphi_{ i } \of{ \gamma } } \lambda_{ q_{ i , k } } \of{ \zeta_{ i } \of{ \gamma } } }{ e^{ \delta \of{ \gamma } + \varphi_{ j } \of{ \gamma } } \lambda_{ q_{ j , k } + 1 } \of{ \zeta_{ j } \of{ \gamma } } } } = \log \of{ \frac{ \lambda_{ q_{ i , k } } \of{ \zeta_{ i } \of{ \gamma } } }{ \lambda_{ q_{ j , k } + 1 } \of{ \zeta_{ j } \of{ \gamma } } } + \varphi_{ i } \of{ \gamma } - \varphi_{ j } \of{ \gamma } } \\
& \geq \of{ 1 - s } \log \of{ \frac{ \lambda_{ q_{ i , k } } \of{ \zeta_{ i } \of{ \gamma } } }{ \lambda_{ q_{ j , k } + 1 } \of{ \zeta_{ j } \of{ \gamma } } } }
\end{align*}
grows at least linearly in translation length $ \norm{ \gamma }_{ \Sigma } $. Thus $ \eta $ is $ P_{ \theta } $-Anosov by \autoref{thm:eigenvalueConfiguration}.

Conversely, now suppose that $ s > 1 $, so that
\[
\log \of{ \frac{ \lambda_{ q_{ i , k } } \of{ \zeta_{ i } \of{ \gamma } } }{ \lambda_{ q_{ j , k } + 1 } \of{ \zeta_{ j } \of{ \gamma } } } } < \varphi_{ j } \of{ \gamma } - \varphi_{ i } \of{ \gamma }
\]
for some $ 1 \leq i , j \leq m $ and $ k \in \theta $ so that $ q_{ i , k } > 0 $ and $ q_{ j , k } < \dim_{ \K } \of{ U_{ j } } $ and some infinite order $ \gamma \in \Gamma $. We see that
\begin{align*}
\log \of{ \frac{ \lambda_{ q_{ i , k } } \of{ \eta_{ i } \of{ \gamma } } }{ \lambda_{ q_{ j , k } + 1 } \of{ \eta_{ j } \of{ \gamma } } } } & = \log \of{ \frac{ \lambda_{ q_{ i , k } } \of{ e^{ \delta \of{ \gamma } + \varphi_{ i } \of{ \gamma } } \zeta_{ i } \of{ \gamma } } }{ \lambda_{ q_{ j , k } + 1 } \of{ e^{ \delta \of{ \gamma } + \varphi_{ j } \of{ \gamma } } \zeta_{ j } \of{ \gamma } } } } = \log \of{ \frac{ e^{ \delta \of{ \gamma } + \varphi_{ i } \of{ \gamma } } \lambda_{ q_{ i , k } } \of{ \zeta_{ i } \of{ \gamma } } }{ e^{ \delta \of{ \gamma } + \varphi_{ j } \of{ \gamma } } \lambda_{ q_{ j , k } + 1 } \of{ \zeta_{ j } \of{ \gamma } } } } \\
& = \log \of{ \frac{ \lambda_{ q_{ i , k } } \of{ \zeta_{ i } \of{ \gamma } } }{ \lambda_{ q_{ j , k } + 1 } \of{ \zeta_{ j } \of{ \gamma } } } } + \varphi_{ i } \of{ \gamma } - \varphi_{ j } \of{ \gamma } \\
& = \log \of{ \frac{ \lambda_{ q_{ i , k } } \of{ \zeta_{ i } \of{ \gamma } } }{ \lambda_{ q_{ j , k } + 1 } \of{ \zeta_{ j } \of{ \gamma } } } } - \of{ \varphi_{ j } \of{ \gamma } - \varphi_{ i } \of{ \gamma } } < 0 .
\end{align*}
If $ \eta $ were $ P_{ \theta } $-Anosov, then this would imply that it has a different large eigenvalue $ \theta $-configuration relative to $ \mathscr{ U } $ than does its block normalization $ \zeta $ relative to $ \mathscr{ U } $. This is impossible, as argued in the proof of \autoref{thm:domain}. Thus this contradiction implies that $ \eta $ is not $ P_{ \theta } $-Anosov.

We have shown that $ \eta $ is $ P_{ \theta } $-Anosov whenever $ s < 1 $ and that $ \eta $ is not $ P_{ \theta } $-Anosov whenever $ s > 1 $. It remains to investigate the case $ s = 1 $. In this case, note that since $ s $ is a homogeneous function of $ \varphi $, $ \eta $ is a limit of representations which by the above arguments are not $ P_{ \theta } $-Anosov. \autoref{thm:stability}, the stability of the $ P_{ \theta } $-Anosov condition, implies that $ \eta $ is not $ P_{ \theta } $-Anosov in this case. \qedhere

\end{proof}

If the block normalization of the block diagonalization of a block upper triangular representation is Anosov, then \autoref{thm:AnosovCharacterization} gives a complete characterization of when said block upper triangular representation is Anosov. However, this characterization relies on the computation of countably infinitely many numbers, as can be seen from the index set of the above supremum.

On the other hand, many of the linear inequalities enforced by such conditions are redundant, as
\[
\frac{ \varphi_{ j } \of{ \gamma } - \varphi_{ i } \of{ \gamma } }{ \log \of{ \frac{ \lambda_{ q_{ i , k } } \of{ \zeta_{ i } \of{ \gamma } } }{ \lambda_{ q_{ j , k } + 1 } \of{ \zeta_{ j } \of{ \gamma } } } } }
\]
is a homogeneous class function of $ \gamma $ for fixed $ 1 \leq i , j \leq m $ and $ k \in \theta $ so that $ q_{ i , k } > 0 $ and $ q_{ j , k } < \dim_{ \K } \of{ U_{ j } } $. Thus the supremum in \autoref{thm:AnosovCharacterization} can instead be indexed with primitive elements in the torsion-free part of the abelianization of $ \Gamma $. While this may still result in a countably infinite indexing set, it begs the question of whether it is possible to achieve the same result by checking only finitely many elements.

\begin{question+} \label{que:polytope}

If the representation $ \zeta \colon \Gamma \to N_{ \mathscr{ U } } $ is $ P_{ \theta } $-Anosov, is $ A^{ \mathscr{ U } }_{ \theta } \of{ \zeta } $ a (finite-sided) polytope? Can the supremum in \autoref{thm:AnosovCharacterization} be realized as a maximum over finitely many $ \gamma \in \Gamma $?

\end{question+}

Preliminary work towards explicitly computing the domains $ A^{ \mathscr{ U } }_{ \theta } \of{ \zeta } $ in cases of limited complexity suggest that it may be possible that the answers to \autoref{que:polytope} are affirmative in some cases.

\printbibliography

\end{document}